\documentclass[a4paper,letterpaper,leqno]{amsart}

\usepackage{amsmath}
\usepackage{amsthm}
\usepackage{amssymb}
\usepackage[english]{babel}
\usepackage{url}
\usepackage[all,cmtip]{xy}
\usepackage{color}
\usepackage{hyperref}
\usepackage{enumerate}
\usepackage{array}
\usepackage{nicefrac}
\usepackage[section]{placeins}

\DeclareMathAlphabet{\mathfr}{U}{euf}{m}{n}

\theoremstyle{plain}
\newtheorem{theorem}{Theorem}[section]

\newtheorem{conjecture}[theorem]{Conjecture}
\newtheorem{proposition}[theorem]{Proposition}
\newtheorem{corollary}[theorem]{Corollary}

\newtheorem{lemma}[theorem]{Lemma}

\theoremstyle{remark}
\newtheorem{remark}{Remark}[section]
\newtheorem{definition}[remark]{Definition}

\numberwithin{equation}{section}

\newcommand{\Q}{\mathbb Q}
\newcommand{\Qbar}{{\overline{\mathbb Q}}}
\newcommand{\Gal}{\mathrm{Gal}}
\newcommand{\R}{\mathbb R}
\newcommand{\Z}{\mathbb Z}

\newcommand{\F}{\mathbb F}
\newcommand{\gM}{\mathfrak N}
\newcommand{\C}{\mathbb C}

\newcommand{\GL}{\mathrm{GL}}

\newcommand{\End}{\operatorname{End}}
\newcommand{\Hom}{\operatorname{Hom}}
\newcommand{\Frob}{\operatorname{Frob}}
\newcommand{\Cl}{\operatorname{Conj}}
\newcommand{\Aut}{\operatorname{Aut}}
\newcommand{\ord}{\operatorname{ord}}
\newcommand{\Sym}{\operatorname{Sym}}
\newcommand{\Spec}{\operatorname{Spec}}

\newcommand{\p}{\mathfrak{p}}
\newcommand{\Res}{\operatorname{Res}}
\newcommand{\Tr}{\operatorname{Tr}}
\newcommand{\im}{\mathrm{Im}}
\newcommand{\diag}{\mathrm{diag}}
\newcommand{\GSp}{\mathrm{GSp}}
\newcommand{\Sp}{\mathrm{Sp}}
\newcommand{\SU}{\mathrm{SU}}
\newcommand{\USp}{\mathrm{USp}}
\newcommand{\MT}{\mathrm{MT}}
\newcommand{\ST}{\mathrm{ST}}
\newcommand{\Hg}{\mathrm{Hg}}

\newcommand{\Lef}{\operatorname{L}}
\newcommand{\TL}{\operatorname{TL}}
\newcommand{\AST}{\operatorname{AST}}
\newcommand{\et}{\operatorname{et}}

\newcommand{\Unitary}{\mathrm{U}}

\begin{document}
\title{Equidistribution, $L$-functions, and Sato-Tate groups}
\date{\today}
\author{Francesc Fit\'e}

\begin{abstract}
In this survey note we present an approach to the generalization of Serre of the Sato-Tate Conjecture. The reader interested in a complete account is referred to the original references \cite{Ser68} and \cite{Ser12}. However, the present note may still be of interest, since we provide a few new examples and supply references to recent progress developed in this area.
\end{abstract}

\maketitle
\tableofcontents

\section{Introduction}

Equidistribution questions arise naturally in number theory and arithmetic geometry. The Dirichlet Theorem on arithmetic progressions, the Cebotarev Density Theorem, or the Sato-Tate Conjecture are well-known examples of them. In this survey note we aim to present an approach to a general framework, due to Serre, where equidistribution questions can be uniformly understood. There is no better option for the reader than to directly look at the original works \cite{Ser68} and \cite{Ser12}. However, the present note may still be of interest, since we supply references to recent progress developed in this field, and we have gathered a collection of a few new examples.

The (classical) Sato-Tate Conjecture concerns the Frobenius distribution of an elliptic curve $E$ defined over $\Q$ without complex multiplication. More precisely, for each prime $p$ of good reduction for $E$, define the quantity
$$
\overline a_p:=\frac{p+1-N_p(E)}{\sqrt p}\,.
$$
Here $N_p(E)$ denotes the number of points defined over $\F_p$ of the reduction of $E$ at $p$. Recall that the Hasse-Weil bound asserts that $|\overline a_p|\leq 2 $, and it thus seems a natural question to ask about the distribution of the \emph{normalized trace} $\overline a_p$ on the interval $[-2,2]$.
In 1968, Sato and Tate independently conjectured that the normalized $\overline a_p$'s are equidistributed on the interval $[-2,2]$ according to the semicircular measure
$$
\frac{1}{2\pi} \sqrt{4-z^2} dz\,.
$$ 
This is today known to hold true (see \cite{BLGHT11}). 

The \emph{generalized Sato-Tate conjecture}, as formulated by Serre (see \cite{Ser12} and \cite{Ser95} for the $\ell$-adic and motivic versions respectively), is a vast generalization of the statement of Sato and Tate, which in particular implies the Dirichlet theorem on arithmetic progressions and the Cebotarev density theorem, and which provides a general and conceptual frame in which equidistribution staments of (cohomological) data attached to arithmetico-geometric objects can be understood. Besides from the references cited above, it is worth mentioning that many of the features of the theory are already outlined in a letter \cite{Ser66} from Serre to Borel of 1966. 

The content of this note goes as follows. In \S\ref{section: equiv and L} we first recall the general strategy to prove equidistribution results (see \cite{Ser68}): one first identifies a compact group in terms of the conjugacy classes of which the equidistribution statement can be formulated; then one establishes the holomorphicity and nonvanishing at $s=1$ of certain $L$-functions attached to the nontrivial irreducible representations of the compact group. In fact, one could add one more step, in the sense that the holomorphicity and nonvanishing at $s=1$ of these $L$-functions is obtained by comparing them to the $L$-functions of a certain family of automorphic forms, for which the holomorphicity and nonvanishing at $s=1$ is known to hold true. Once the general strategy has been introduced, in \S\ref{subsection: examples L} we analyse four examples where equidistribution has been proved by means of a successful application of the method described. We also provide one example where equidistribution is just conjectural. 

In \S\ref{section: def and conj}, we define the Sato-Tate group of a smooth and projective variety and state the generalized Sato-Tate Conjecture. We thus restrict to a particular case of \cite[Chap. 8]{Ser12}, where, more generally, the case of a scheme of finite type over~$\Z$ is considered. We then revisit some of the examples of \S\ref{section: equiv and L} from this perspective. We review the connection between the Sato-Tate group and the Mumford-Tate group and we close the section by considering the case of absolutely simple abelian varieties with complex multiplication. This example illustrates the whole theory developed in \S\ref{section: equiv and L} and \S\ref{section: def and conj}.

In \ref{section: GT results}, we describe recent results on the classification of Sato-Tate groups for families of selfdual motives with rational coefficients and fixed weight and Hodge numbers (such as, for example, abelian surfaces defined over a number field). 
Much of the interest in this kind of classifications arised when computational methods permitted to perform numerical tests of the generalized Sato-Tate Conjecture for curves of low genus (see \cite{KS09}). We will conclude the section by gathering references to several results in this area, that have been obtained in the last years.  

Finally, let us just mention that Bucur and Kedlaya (see \cite{BK12}) have used an effective\footnote{By \emph{effective}, we mean here that there is an effective bound on the asymptotic error term implicit in any equidistribution statement.} form of the generalized Sato-Tate Conjecture to obtain an interesting arithmetic application: an effective upper bound for the smallest prime at which two non-isogenous elliptic curves have distinct Frobenius trace (conjectural to the generalized Riemann Hypothesis). Despite the beauty of the theory and the range of possibilities that it offers, we will not say anything else about the effective form of the generalized Sato-Tate Conjecture in this note.

\textbf{Acknowledgements.} A first draft of this note was used for a series of talks at ``Seminari de Teoria de Nombres de Barcelona" in February 2013. A second version was written for the occasion of the conference ``Quintas Jornadas de Teor\'ia de N\'umeros" held in Sevilla in July 2013 (although my talk in that occasion focused on \cite{FGL13}). 
This note also served as the basis for a series of three talks at the Winter School ``Frobenius distributions of curves" held in Luminy in February 2014. Thanks to the organizers of all three meetings. I am thankful to Xevi Guitart, Joan-C. Lario, and Andrew V. Sutherland for helpful comments, and I am gratefully indebted to the referee and to Jean-Pierre Serre for their numerous corrections. The author was funded by the German Research Council via CRC 701 and was also partially supported by MECD project MTM2012-34611. This work is devoted to the memory of Anna Maria Argerich.

\section{Equidistribution and $L$-functions}\label{section: equiv and L}

In this section, we first formalize the notion of equidistribution. Then, in the case of a compact group (our case of interest in all future sections), we show how equidistribution is connected to the holomorphicity and nonvanishing of certain $L$-functions attached to the nontrivial irreducible representations of the group. The exposition closely follows \cite[Chap.I]{Ser68}. This is still the general strategy that one follows to prove equidistribution results.

\subsection{Definition of equidistribution}

Let $X$ be a compact topological space and denote by $C(X)$ the Banach space of continuous, complex valued functions $f$ on $X$, with norm $||f||:=\sup_{x\in X}|f(x)|$. Let 
$$
\mu\colon C(X)\rightarrow \C
$$ 
be a Radon measure on $X$, that is, a continuous linear form on $C(X)$. Sometimes the integral notation is preferred, and the image $\mu(f)$ of $f\in C(X)$ by $\mu$ is  denoted by $\int_{x\in X}f(x)\mu(x)$. We require $\mu$ to be:
\begin{itemize} 
\item positive (i.e. if $f$ is real and positive, then so is $\mu(f)$); and 
\item of total mass $1$ (i.e. the image by $\mu$ of the constant function equal to~$1$ is~$1$).
\end{itemize}

\begin{definition}
Let $\{x_n\}_{n\geq 1}$ be a sequence of points of $X$. The sequence $\{x_n\}_{n\geq 1}$ is said to be equidistributed over $X$ with respect to $\mu$ (or, simply $\mu$-equidistributed) if for every $f\in C(X)$ we have
$$\mu(f)=\lim_{n\rightarrow \infty}\frac{1}{n}\sum_{i=1}^nf(x_i)\,.$$
\end{definition}

Note that if $\{x_n\}_{n\geq 1}$ is equidistributed with respect to $\mu$, then $\mu$ must be positive and of total mass $1$.

\subsection{The case of a compact group}\label{section: compact group}
Let $G$ be a compact group and let $X$ denote the set of conjugacy classes of $G$. Let $\mu$ be the Haar measure of $G$ (with $\mu(G)=1$). Denote also by $\mu$ the push-forward of $\mu$ on $X$, that is, if $\pi:G\rightarrow X$ denotes the natural projection, then define
$$
\mu(f):=\mu(f\circ \pi), \qquad\text{ for all $f\in C(X)$\,.}
$$

\begin{proposition}\label{proposition: equidistribution}
A sequence $\{x_n\}_{n\geq 1}$ of elements of $X$ is $\mu$-equidistributed if and only if for every irreducible nontrivial character $\chi$ of $G$, we have 
$$\lim_{n\rightarrow\infty}\frac{1}{n}\sum_{i=1} ^n\chi(x_i)=0\,.$$
\end{proposition}

\begin{proof} It is enough to check
$$\mu(\chi_a)=\lim_{n\rightarrow \infty}\frac{1}{n}\sum_{i=1}^n\chi_a(x_i)\,,\qquad \text{for every $a\in \mathcal A$,}$$
where $\{\chi_a\}_{a\in \mathcal A}$ is a family of continuous functions on $X$ such that their linear combinations are dense in $C(X)$. By the Peter-Weyl theorem, the set of irreducible characters $\chi$ of $G$ constitute such a family. The proposition then follows from the fact that $\mu(1)=1$ and $\mu(\chi)=0$ if $\chi$ is irreducible and nontrivial.
\end{proof}

\subsubsection{One application}\label{section: varyingfield} Let $E/\F_q$ be an elliptic curve defined over the finite field of $q=p^m$ elements. 
For $n\geq 1$, define the quantity
$$
\overline a_{q^ n}:=\frac{q^n+1-N_{q^n}(E)}{q^{n/2}}\,.
$$
Here $N_{q^n}(E)$ denotes the number of points of $E$ defined over $\F_{q^n}$.

\begin{proposition}\label{proposition: varyingfield} If $E/\F_q$ is ordinary, then the sequence $\{\overline a_{q^n}\}_{n\geq 1}$ is equidistributed on $[-2,2]$ with respect to the measure
\begin{equation}\label{equation: measure}
\frac{1}{\pi}\frac{dz}{\sqrt{4-z^2}}\,,
\end{equation}
where $dz$ is the restriction of the Lebesgue measure on $[-2,2]$.
\end{proposition}

\begin{proof}
Recall that there exists $\alpha\in \C$ of absolute value $q^{1/2}$ such that $\overline a_{q^n}=(\alpha^ n+\overline \alpha ^ n)/q^{n/2}$. 
Let $\Unitary(1):=\{u\in \C^*\,|\, u\overline u =1\}$ be the unitary group  
of degree one, that is, the group of complex numbers of absolute value $1$. Since the projection of the Haar measure of $\Unitary(1)$ on the interval $[-2,2]$ by the map $u\mapsto z:=u+\overline u$ is the measure (\ref{equation: measure}), it suffices to show that the sequence $\{\alpha^n/q^{n/2}\}_{n\geq 1}$ is equidistributed on $\Unitary (1)$ with respect to the Haar measure. Observe that the nontrivial irreducible characters of $\Unitary(1)$ are of the form 
\begin{equation}\label{equation: phia}
\phi_a\colon \Unitary(1)\rightarrow \C^*\\,\qquad \phi_a(u):=u^ a\,,
\end{equation}
for some $a\in \Z^ *$. We conclude by applying Proposition \ref{proposition: equidistribution}, after noting that
\begin{equation}\label{equation: geosum}
\lim_{n\rightarrow\infty}\frac{1}{n}\sum_{i=1} ^n\frac{\alpha^{ia}}{q^{ia/2}}=\lim_{n\rightarrow\infty}\frac{1}{n}\frac{\alpha^{a(n+1)}/q^{a(n+1)/2}-\alpha^a/q^{a/2}}{\alpha^{a}/q^{a/2}-1}=0\,.
\end{equation}
Observe that the hypothesis of $E$ being ordinary ensures that $\alpha/q^{1/2}$ is not a root of unity, and thus $\alpha^{a}/q^{a/2}-1$ is nonzero for every $a\in \Z^*$. 
\end{proof}

See \cite{AS12} for analogous results for smooth projective curves $C/\F_q$ of arbitrary genus $g\geq 1$ whose Frobenius eigenvalues satisfy a certain \emph{neatness} condition\footnote{It is precisely the neatness condition what permits to reproduce the computation of (\ref{equation: geosum}) in the general case, obtaining a geometric series whose ratio can be guaranteed not to be 1.} of multiplicative independence. Zarhin has completely characterized the neat abelian varieties of dimension $g\leq 3$ (see \cite{Zar14}). 

\subsection{The connection with $L$-functions}\label{section: connL}

Let $P$ be the set of primes of a number field $F$ and let $S$ be a finite subset of $P$. Assume we are given an ordering $\{\p_i\}_{i\geq 1}$ of $P\setminus S$ by norm, that is, $N(\p_i)\leq N(\p_j)$ if and only if $i\leq j$, where $N(\p)$ denotes the absolute norm of $\p$. Let $G$ be a compact group, $X$ the set of its conjugacy classes, and $\{x_{\p_i}\}_{i\geq 1}$ a sequence of elements of $X$ indexed by $\{\p_i\}_{i\geq 1}$ (we are thus in the setting of \S\ref{section: compact group}). As before, let $\mu$ denote the Haar measure of $G$.

For $\varrho\colon G\rightarrow \GL_d(\C)$ a continuous irreducible representation of $G$, define the Euler product
\begin{equation}\label{equation: Lfunc}
L(\varrho,s):=\prod_{i\geq 1}\det(1-\varrho(x_{\p_i}) N(\p_i)^{-s})^{-1}
\end{equation}
for a complex $s\in \C$ with $\Re(s)> 1$. 

\begin{theorem}\label{theorem: equidistribution} For a given sequence $\{x_{\p_i}\}_{i\geq 1}\subseteq X$ and for any irreducible representation $\varrho$ of $G$, assume that $L(\varrho,s)$ is meromorphic for $\Re(s)\geq 1$ and has no zero and no pole in this halfplane except possibly at $s=1$. Then, the sequence $\{x_{\p_i}\}_{i\geq 1}$ is $\mu$-equidistributed over $X$ if and only if, for every irreducible nontrivial representation $\varrho$ of $G$, the Euler product $L(\varrho,s)$ extends to a holomorphic function on $\Re(s)\geq 1$ and is nonvanishing at $s=1$.
\end{theorem}

\begin{proof}
Let $\chi$ denote the character of $\varrho$. We will prove the following claim: $L(\varrho,s)$ extends to a holomorphic and nonvanishing function on $\Re(s)\geq 1$ if and only if
$$
\sum_{N(\p_i)\leq n}\chi(x_{\p_i})=o\left( \frac{n}{\log(n)}\right)\,,\qquad n\rightarrow \infty\,.
$$
Then the theorem follows from Proposition \ref{proposition: equidistribution} and the Prime Number Theorem for the number field $F$, which ensures that the number of $\p_i$'s with $N(\p_i)\leq n$ is equivalent to $n/\log(n)$ when $n\rightarrow \infty$. In order to prove the claim, write
$$
L(\varrho,s)=\prod_{i\geq 1}\prod_{j=1}^d\frac{1}{1-\lambda_{i,j}N(\p_i)^{-s}}\,,
$$
where $\lambda_{i,j}$ for $j=1,\dots, n$ are the eigenvalues of $\varrho(x_{\p_i})$. The logarithmic derivative of $L(\varrho,s)$ is
$$
\frac{L'(\varrho,s)}{L(\varrho,s)}=-\sum_{i\geq 1}\sum_{j=1}^d\sum_{m\geq 1}\frac{\lambda_{i,j}^m\log(N(\p_i))}{N(\p_i)^{ms}}=-\sum_{i\geq 1}\sum_{m\geq 1}\frac{\chi(x_{\p_i}^m)\log(N(\p_i))}{N(\p_i)^{ms}}\,.
$$
Since 
$$
\sum_{i\geq 1}\sum_
{m \geq 2}\frac{\log(N(\p_i))}{|N(\p_i)^{ms}|}
$$
converges for $\Re(s)> 1/2$, we can write
\begin{equation}\label{equation: logder}
\frac{L'(\varrho,s)}{L(\varrho,s)}=F(s) +\phi(s)\,,
\end{equation}
where $\phi(s)$ is holomorphic for $\Re(s)>1/2$ and $F(s)= -\sum_{i\geq 1}\frac{\chi(x_{\p_i})\log(N(\p_i))}{N(\p_i)^{s}}$. By hypothesis, $L(\varrho,s)$ is meromorphic for $\Re(s)\geq 1$ and has no zero and no pole except for (possibly) a zero at $s=1$, say of order\footnote{By this we mean that $L(\varrho,s)$ has a pole (resp. a zero) of order $c$ (resp. $-c$) at $s=1$.} $-c$. Then $L'(\varrho,s)/L(\varrho,s)$ is meromorphic for $\Re(s)\geq 1$ with at most one simple pole at $s=1$ with residue~$c$. Since $\phi(s)$ is holomorphic for $\Re(s)>1/2$, (\ref{equation: logder}) shows that $F(s)$ is also meromorphic for $\Re(s)\geq 1$ with at most one simple pole at $s=1$ with residue~$c$. Then, the Wiener-Ikehara Theorem applied to $F(s)$ implies that
$$
\sum_{N(\p_i)\leq n}\chi(\p_i)\log(N(\p_i))=cn+o(n)\,,\qquad n\rightarrow \infty\,.
$$
Finally, the Abel summation trick yields
$$
\sum_{N(\p_i)\leq n}\chi(x_{\p_i})=c\frac{n}{\log(n)}+o\left(\frac{n}{\log(n)}\right)\,,\qquad n\rightarrow \infty\,,
$$
from which the claim follows.
\end{proof}

\subsection{Some known cases and a conjectural example}\label{subsection: examples L}

Throughout this section $F$ denotes a number field, $P$ its set of primes, and $\{\p_{i} \}_{i\geq 1}$ an ordering by norm of $P\setminus S$. We will consider four examples in which:
\begin{itemize}
\item we specify $G$, $X$, $\mu$, $S$, and the sequence $\{x_{\p_i}\}_{i\geq 1}$; 
\item the nonvanishing $L$-function condition of Theorem \ref{theorem: equidistribution} is known to be true. 
\end{itemize}
In other words, we consider four examples in which equidistribution of the sequence $\{x_{\p_i}\}_{i\geq 1}$ with respect to the Haar measure of $G$ is known to be true.
To conclude we will present a conjectural example.

\subsubsection{Cebotarev Density Theorem}\label{section: Chebotarev} 

Let $G$ be the Galois group of a finite Galois extension $L/F$ and $X$ the set of conjugacy classes of $G$. Then $\mu$ is just the discrete measure on $X$ (with masses corresponding to the size of each conjugacy class). Let $S$ be the set of ramified primes. For $\p\not\in S$, let $x_\mathfrak p$ be the Frobenius element $\Frob_\p\in X$ at $\p$.

A representation $\varrho$ of $\Gal(L/F)$ is called an Artin representation, and one may associate to $\varrho$ an Artin $L$-function (see \cite[Chap. VII]{Neu92}). If they have dimension~ $1$, we can identify them, via Artin reciprocity, with unitarized Hecke characters, and we refer to their associated $L$-functions by Hecke $L$-functions. We remark that the $L$-function $L(\varrho,s)$ defined by (\ref{equation: Lfunc}) differs from the Artin $L$-function attached to $\varrho$ in only a finite number of holomorphic and nonvanishing Euler factors. 

\begin{theorem}[Hecke]\label{theorem: Hecke} The Hecke $L$-function of a nontrivial unitarized Hecke character is holomorphic and nonvanishing for $\Re(s)\geq 1$.
\end{theorem}

\begin{proof}
See Hecke's historical reference \cite{Hec20} or \cite[Chap. XV]{La94}.
\end{proof}

\begin{theorem}\label{theorem: Art L} The Artin $L$-function of an irreducible nontrivial Artin representation is holomorphic and nonvanishing\footnote{In fact, Artin's Conjecture states that the Artin $L$-function of a nontrivial irreducible Artin representation is holomorphic on the whole complex plane. Artin proved the conjecture for representations of dimension 1, and several cases of dimension 2 with $F=\Q$ are also known (see \cite[\S10]{KW09}).} for $\Re (s)\geq 1$.
\end{theorem}

\begin{proof} By Theorem \ref{theorem: Hecke}, the $L$-function of a nontrivial unitarized Hecke character is holomorphic and nonvanishing for $\Re(s)\geq 1$. Brauer's theorem on induced characters implies that Artin $L$-functions are products of positive and negative integral powers of Hecke $L$-functions. Therefore, Artin $L$-functions are also holomorphic and nonvanishing for $\Re(s)\geq 1$.
\end{proof}

Theorems \ref{theorem: equidistribution} and \ref{theorem: Art L}, imply that the sequence $\{x_{\p_i}\}_{i\geq 1}$ is equidistributed on $X$. Since $\mu$ is the discrete measure on $X$, we recover the Cebotarev Density Theorem.

\begin{corollary}[Cebotarev Density Theorem]\label{corollary: Cebotarev} Let $c$ be a conjugacy class of $\Gal(L/F)$. Then the density of the set of primes $\p$ such that $\Frob_\p= c$ is
$
\frac{\#c}{[L:F]}.
$  
\end{corollary}

\subsubsection{Equidistribution of eigenvalues of algebraic Hecke characters attached to a quadratic imaginary field $K$}\label{section: CM ST} 
In this section, we will look at a particular class of Hecke characters, that is, those that are algebraic and are attached to an imaginary quadratic field~$K$ (throughout this section we have $F=K$). Reviewing their connection with CM modular forms, we will see how the Sato-Tate Conjecture for CM elliptic curves defined over $K$, follows from Theorem \ref{theorem: Hecke}. 

Let $\gM$ be an ideal of $K$, and $l\geq 1$. Let $I_{\gM}$ stand for the group of fractional ideals of $K$ coprime to $\gM$. An algebraic Hecke character\footnote{Via Artin reciprocity and unitarizing it so that it has image in $\Unitary(1)$, we may see $\psi$ as a character of $\Gal(L/K)$, where $L$ is the ray class field of $K$ of modulo $\gM$. We thus see that the notion of Hecke character in this section is compatible with that in \S\ref{section: Chebotarev}.} of $K$ of modulo $\gM$ and type at infinity $l$ is a homomorphism
\begin{equation}\label{equation: Hecke char}
\psi\colon I_{\gM}\rightarrow \C^*
\end{equation}
such that $\psi(\alpha\mathcal O_K)=\alpha^l$ for all $\alpha\in K^ *$ with $\alpha\equiv^\times 1 \pmod{\gM}$\footnote{For a number field $L$, an ideal $\gM$ of the ring of integers of $L$, and $\alpha_1,\alpha_2\in L^*$, we write $\alpha_1 \equiv ^\times \alpha_2 \pmod \gM$ if $\alpha_1$ and $\alpha_2$ are multiplicatively congruent modulo $\gM$, i.e., if $v(\alpha_2/\alpha_1 -1)\geq v(\gM)$ for every discret valuation~$v$ of~$L$.}.
We extend~$\psi$ by defining it to be $0$ for all fractional ideals of $K$ that are not coprime to $\gM$. We say that $\gM$ is the conductor of $\psi$ if the following holds: if $\psi$ is defined modulo $\gM'$, then $\gM|\gM'$. The $L$-function of $\psi$ is then defined by
$$
L(\psi,s):=\prod_{\p}(1-\psi(\p)N(\p)^{-s})^{-1}\,,
$$
where the product runs over all prime ideals of $K$. Let $G$ be the unitary group $\Unitary(1)$. Since it is abelian, we have $X=G$ and since the Haar measure is invariant under translations, we have that $\mu$ is the uniform measure on $\Unitary(1)$.
Let $S$ be the set of primes of $K$ dividing $\gM$. Fix an embedding of $K$ in $\C$ and for $z\in \C$, let $|z|:=\sqrt{z\overline z}$. Then one has that $|\psi(\p)|=N(\p)^{l/2}$. For $\p\not \in S$, define 
$$
x_\p:=\frac{\psi(\p)}{N(\p)^{l/2}}\in X=\Unitary(1)\,.
$$

\begin{corollary}\label{corollary: equid cm}
The sequence $\{x_{\p_i}\}_{i\geq 1}$ is $\mu$-equidistributed on $\Unitary(1)$.
\end{corollary}

\begin{proof} The nontrivial irreducible characters of $\Unitary(1)$ are $\phi_a\colon \Unitary(1)\rightarrow \C^*$ for $a\in \Z^*$ (see the proof of Proposition \ref{proposition: varyingfield}).
By Theorem \ref{theorem: equidistribution}, it is enough to prove that $L(\phi_a,s)$ is holomorphic and nonvanishing for $\Re(s)\geq 1$. But this is a consequence of Theorem \ref{theorem: Hecke}, and the fact that a nontrivial power of $\psi$ is again a nontrivial Hecke character.
\end{proof}

\subsubsection*{First application: Sato-Tate for CM elliptic curves defined over $K$} 

Let $E$ be an elliptic curve with complex multiplication by a quadratic imaginary field $K$ and assume that $E$ is defined over $K$. Let $\gM$ be the conductor of $E$. This is an ideal of the ring of integers of $K$ divisible precisely by the primes at which $E$ has bad reduction\footnote{See for example \cite[Chap. IV, \S10]{Sil94} for a description of the exponent of each prime in the factorization of $\gM$.}. A classical result of Deuring ensures the existence of an algebraic Hecke character $\psi_E$ of~$K$ of conductor~$\gM$ and of type at infinity $1$ that is attached to $E$ by means of the following property.  

For every $\p\nmid \gM$, let $a_\p$ denote the quantity $N(\p)+1-N_\p(E)$, where $N_\p(E)$ denotes the number of points defined over the residue field at $\p$ of the reduction of~$E$ at~$\p$. Then 
\begin{equation}\label{equation: heckedeuchar}
a_\p=\psi_E(\p)+\overline {\psi_E(\p)}\,.
\end{equation}
\begin{corollary}\label{corollary: STCM}
Let $E$ be an elliptic curve with CM by a quadratic imaginary field~$K$ and assume that~$E$ is defined over~$K$. Then, the sequence of normalized traces $\{a_{\p_i}/\sqrt{N(\p_i)}\}_{i\geq 1}$ is equidistributed on $[-2,2]$ with respect to the measure
$$
\mu_{\rm{cm}}:=\frac{1}{\pi}\frac{dz}{\sqrt{4-z^2}}\,.
$$
\end{corollary}

\begin{proof}
Apply Corollary \ref{corollary: equid cm} to deduce that $\{ x_{\p_i}\}_{i\geq 1}$ is equidistributed on $\Unitary(1)$, where $x_\p:=\psi_E(\p)/\sqrt {N(\p)}$.
We have already seen in \S\ref{section: varyingfield} that the projection of the Haar measure $\mu$ of $\Unitary(1)$ on the interval $[-2,2]$ by the trace map is~$\mu_{\rm{cm}}$.
\end{proof}

\subsubsection*{Second application: equidistribution of eigenvalues of CM newforms} 

Let us denote by $S_k(\Gamma_1(M))$ the complex space of weight $k$ cusp forms for $\Gamma_1(M)$. We assume throughout the section that $k\geq 2$. There is a decomposition 
$$
S_k(\Gamma_1(M))=\bigoplus_\varepsilon S_k(\Gamma_0(M),\varepsilon),
$$
where $\varepsilon\colon (\Z/M\Z)^*\rightarrow \C^*$ runs over the characters of $(\Z/M\Z)^*$, and $S_k(\Gamma_0(M),\varepsilon)$ denotes the space of weight $k$ cusp forms for $\Gamma_0(M)$ with Nebentypus $\varepsilon$. We say that 
$$
f=\sum_{n\geq 1} a_n(f)q^n\in S_k(\Gamma_1(M),\varepsilon)\,,\quad\text{ with }q=e^{2\pi i z}\text{ and }\Im(z)> 0\,,
$$
is new if it can not be written as a linear combination of forms of lower levels. We say that $f$ is normalized if $a_1(f)=1$. We will use the term \emph{newform} for a normalized new eigenform for the action of the Hecke algebra. If $f$ is a newform, then $a_p(f)=\overline {a_p(f)}\varepsilon(p)$.

We say that $f$ has complex multiplication (CM) if there exists a Dirichlet character $\chi$ such that $a_p(f)=\chi(p) a_p(f)$ for a set of primes $p$ of density~$1$. In this case, one can show that $\chi$ is quadratic and attached to a quadratic imaginary field~$K$ (see \cite[\S 3]{Rib77}). It is then also common to say that $f$ has CM by the quadratic imaginary field $K$.

We recall a result that states that every CM cusp form comes from an algebraic Hecke character. This can be seen as a generalization of the result of Deuring seen in the previous paragraph. Let $\Delta_K$ denote the absolute value of the discriminant of $K$, and $\chi$ the quadratic character attached to $K$. 
 
\begin{theorem}\cite[Cor. 3.5, Thm. 4.5]{Rib77}\label{theorem: HecShiRib} Let $\psi$ be an algebraic Hecke character of~$K$ of modulo $\gM$ and type at infinity $l$. Then $f_{\psi}:=\sum_{\mathfrak a}\psi(\mathfrak a)q^ {N(\mathfrak a)}$, where the sum runs over integral ideals of $K$, is an eigenform for the action of the Hecke algebra of weight $l+1$, level $\Delta_KN(\gM)$, with CM by $K$, and Nebentypus $\chi\eta$, where
$$
\eta\colon(\Z/N(\gM)\Z)^*\rightarrow \C^*\,,\qquad \eta(n):=\frac{\psi(n\mathcal O_K)}{n^ l}\,. 
$$
Moreover, $f_\psi$ is new at this level if and only if $\gM$ is the conductor of $\psi$.
Conversely, every new eigenform $f\in S_k(\Gamma_1(M))$ with CM by a quadratic imaginary field arises from an algebraic Hecke character $\psi_f$ in this way.
\end{theorem}

Let $f$ be a newform and let $K_f$ denote the number field obtained by adjoining to $\Q$ the Fourier coefficients $a_n(f)$. Denote by $G_\Q$ the absolute Galois group of~$\Q$. As shown by Deligne \cite{Del71}, for each prime ideal $\lambda$ of $K_f$ lying over a prime $\ell$, there is an irreducible $2$-dimensional Galois representation
\begin{equation}\label{equation: lambda}
\varrho_{f,\lambda}: G_\Q\rightarrow \GL_2(K_{f,\lambda})\,,
\end{equation}
where $K_{f,\lambda}$ is the completion of $K_f$ at $\lambda$. For any $p\nmid M\ell$, one has the following fact 
\begin{equation}\label{equation: char poly}
\text{Characteristic polynomial of } \varrho_{f,\lambda}(\Frob_p) = T^2-a_p(f)T+\varepsilon(p)p^{k-1}
\,.
\end{equation}
Here, $\Frob_p$ denotes an absolute Frobenius element at $p$. It follows that 
\begin{equation}\label{equation: detcyc}
\det(\varrho_{f,\lambda})=\varepsilon\chi_\ell^{k-1}\,,
\end{equation} 
where $\chi_\ell$ is the $\ell$-adic cyclotomic character.
Let $\p$ be a prime of $K$ not dividing $M\ell$. In the case that $f$ has CM by $K$, the proof of Theorem \ref{theorem: HecShiRib} (see \cite[Thm. ~4.5]{Rib77}) establishes that the characteristic polynomial of $\varrho_{f,\lambda}(\Frob_\p)$ is
$$
(T-\psi_f(\p))(T-\psi_f( \overline \p))\,.
$$
Then, (\ref{equation: detcyc}) implies that $\psi_f( \overline \p)= \varepsilon(N(\p)) \overline{\psi_f(\p)}$.
Consider the commutative group
$$
G:=\left\{ 
\begin{pmatrix}
u & 0\\
0 & \zeta\overline u
\end{pmatrix} \, |\,\zeta\in\im(\varepsilon),u\in\C^*,|u|=1\right\}\,.
$$ 
Observe that if $\varepsilon$ is trivial, $G$ is isomorphic to the unitary group $\Unitary(1)$ of degree~$1$. Let $\mu$ denote the Haar measure of $G$ and observe that $G$ coincides with the set of its conjugacy classes. Let $S$ be the set of primes of $K$ lying over~$M\ell$. For every $\p\not\in S$, let 
$$
x_{\p}:=\begin{pmatrix}
\psi_f(\p)/N(\p)^{(k-1)/2} & 0\\
0 & \varepsilon(N(\p))\overline {\psi_f(\p)}/N(\p)^{(k-1)/2}
\end{pmatrix}\in X\,.
$$

\begin{proposition}\label{proposition: equid cm mf} The sequence $\{x_{\p_i}\}_{i\geq 1}$ is $\mu$-equidistributed on $X$.
\end{proposition}

\begin{proof}
First note that 
$$
G\simeq\Unitary(1)\times\im(\varepsilon)\qquad \text{and}\qquad \im(\varepsilon)\simeq \Gal(K_\varepsilon/K)\,,
$$ 
where $K_\varepsilon/K$ is a cyclic extension of order $\#\im(\varepsilon)$. Thus the nontrivial irreducible representations of $G$ are $\phi_a\otimes \chi$, where $\phi_a(u)=u^ a$ for some $a\in \Z$ and $\chi$ is a character of $\Gal(K_\varepsilon/K)$ such that either $a\not =0$ or $\chi$ is nontrivial.
By Theorem~\ref{theorem: equidistribution}, it is enough to prove that $L(\phi_a\otimes \chi,s)$ is holomorphic and nonvanishing for $\Re(s)\geq 1$. But this is a consequence of Theorem \ref{theorem: Hecke}, and the fact that $\psi_f^a\otimes \chi$ is again a nontrivial Hecke character.
\end{proof}

\subsubsection{Equidistribution of eigenvalues of non-CM newforms}\label{section: classic ST} 

Let $f\in S_k(\Gamma_1(M))$ be a newform without complex multiplication. As in the previous section assume $k\geq 2$. Recall that we may attach to it a $\lambda$-adic irreducible Galois representation as in (\ref{equation: lambda}).
Consider the group
\begin{equation}\label{equation: STgmf}
G:=\left\{ 
\begin{pmatrix}
a & b\\
-\overline b & \overline a
\end{pmatrix} \, |\,a,b\in\C,\, a\overline a+b\overline b\in\im(\varepsilon)
\right\}\,,
\end{equation} 
and let $\mu$ be its Haar measure. Observe that if $\varepsilon$ is trivial, $G$ coincides with the symplectic unitary group\footnote{Note that $\USp(2)=\SU(2)$, since a unitary matrix of determinant~1 of degree~2 is automatically symplectic.} $\USp(2)$ of degree~$2$. Let $X$ denote the set of conjugacy classes of $G$. In this section $F=\Q$, and $S$ denotes the set of primes dividing~$M\ell$. For every $p\not\in S$, let $x_{p}$ denote the conjugacy class of $X$ defined by
$$
\frac{\varrho_{f,\lambda}(\Frob_{p})}{p^{(k-1)/2}}\,.
$$
The fact, that $x_p$ defines a conjugacy class in $X$, follows from ~(\ref{equation: char poly}) and the Ramanujan-Petersson inequality asserting that $|a_p(f)|\leq 2p^{(k-1)/2}$ (see \cite[Th\'eor\`eme 8.2]{Del74}).

\begin{theorem}[Barnet-Lamb, Geraghty, Harris, Taylor]\label{theorem: ST}
Let $f\in S_k(\Gamma_0(M),\varepsilon)$ be a newform of weight $k\geq 2$, level $M$, and Nebentypus $\varepsilon: (\Z/M\Z)^ *\rightarrow \C^ *$. For $D\geq 1$, let $\chi\colon (\Z/D\Z)^*\rightarrow \C^*$ be a Dirichlet character. Write
$$
f(z)=\sum_{n\geq 1} a_n(f) q^n,\qquad \text{with }a_1(f)=1\text{ and }q=e^ {2 \pi i z}.
$$
For $p\nmid  M$ write $\{\alpha_pp^{(k-1)/2},\beta_pp^{(k-1)/2}\}$ for the roots of 
$$
T^2-a_p(f)T+\varepsilon(p)p^{k-1}
$$
Then, if $f$ does not have complex multiplication, for $m\geq 1$, the Euler product 
$$
\prod_{p\nmid MD}\prod_{i=0}^m(1-\chi(p)\alpha_p^{m-i}\beta_p^{i}p^{-s})^{-1}
$$
has meromorphic continuation to the whole complex plane, and is holomorphic and nonvanishing for $\Re(s)\geq 1$. 
\end{theorem}

\begin{proof}
See \cite[Theorem B]{BLGHT11}.
\end{proof}

\begin{corollary}\label{corollary: equid no cm} The sequence $\{x_{p_i}\}_{i\geq 1}$ is $\mu$-equidistributed on $X$.
\end{corollary}

\begin{proof}
It is easily checked that the irreducible representations of~$G$ are of the form $\Sym^m(\C^ 2)\otimes \chi$, where $\chi$ is a character of the cyclic group $\im(\varepsilon)$. By Theorem~\ref{theorem: equidistribution}, it suffices to show that for $m\geq 1$, the $L$-function $L(\Sym^m(\C^2)\otimes \chi,s)$
is holomorphic and nonvanishing for $\Re(s)\geq 1$ (the case $m=0$ is covered by Dirichlet's Theorem). But this follows from  Theorem~\ref{theorem: ST}, and the following two observations
\begin{itemize} 
\item Note that by (\ref{equation: char poly}), if $\{\alpha_pp^{(k-1)/2},\beta_pp^{(k-1)/2}\}$ are the roots of 
$
T^2-a_p(f)T+\varepsilon(p)p^{k-1}
$,
then $\{\alpha_p, \beta_p\}$ are the eigenvalues of $x_p$. 
\item Recall that if $e_1,\dots,e_n$ is a basis for the vector space $V$, then $\{e_{i_1}\cdot{}_{\dots}\cdot e_{i_m}|1\leq  i_1\leq i_2\leq {}_{\dots} \leq i_m\leq n\}$
is a basis of $\Sym^m (V)$.
\end{itemize}
\end{proof}

\subsubsection*{Application: The classical Sato-Tate Conjecture for non CM elliptic curves defined over $\Q$}\label{section: non CM} 

Let $E$ be an elliptic curve defined over $\Q$ without complex multiplication and with conductor $M$. For a prime $\ell$, let $V_\ell(E)$ denote its $\ell$-adic (rational) Tate module. Attached to $E$, there is an $\ell$-adic representation
\begin{equation}\label{equation: ladicrep}
\varrho_{E,\ell}\colon G_\Q \rightarrow \Aut(V_\ell(E))\,,
\end{equation}
given by the action of $G_\Q$ on $V_\ell(E)$. 
Let $S$ be the set of primes dividing~$M\ell$, $G:=\SU(2)$, $\mu$ its Haar measure, and $X$ the set of conjugacy classes of $G$. For every $p\not\in S$, let 
$$
x_{p}:=\frac{\varrho_{E,\ell}(\Frob_{p})}{\sqrt p}\,.
$$
The Hasse-Weil bound (together with the fact that $\det(\varrho_{E,\ell})=\chi_\ell$) implies that~$x_p$ indeed defines a conjugacy class in $X$. Recall that the Modularity Theorem states that there exists a newform $f\in S_2(\Gamma_0(M))$ with rational coefficients (i.e, with $K_f=\Q$) such that $\varrho_{f,\ell}\simeq \varrho_{E,\ell}$.

\begin{corollary}
Let $E$ be an elliptic curve defined over $\Q$ without CM. Then, the sequence of normalized traces $\{a_{p_i}/\sqrt{p_i}\}_{i\geq 1}$ is equidistributed on $[-2,2]$ with respect to the measure
\begin{equation}\label{equation: mu no cm}
\mu_{\rm{ST}}:=\frac{1}{2\pi} \sqrt{4-z^2} dz\,.
\end{equation}
\end{corollary}

\begin{proof}
By Corollary~\ref{corollary: equid no cm}, the sequence of $\{x_{p_i}\}_{i\geq 1}$ is $\mu$-equidistributed over~$X$. It is then enough to observe that the projection of the measure $\mu$ on the set of conjugacy classes of $\SU(2)$ by the trace map on $[-2,2]$ is $\mu_{\rm{ST}}$.
\end{proof}

\subsubsection{A Cebotarev-Sato-Tate Density Theorem}

In \S\ref{section: CM ST} and \S\ref{section: classic ST}, we have already implicitly seen hybrid forms of the Cebotarev Density Theorem and the Sato-Tate Conjecture in the case that the extension of fields involved is abelian. In this section, we recall analogous results for general Galois extensions.

Let $E$ be an elliptic curve defined over $\Q$ without complex multiplication. Let $F=\Q$ and $L/\Q$ be a finite Galois extension. Let $G:=\SU(2)\times \Gal(L/\Q)$ and let~$S$ be the set of primes made up by $\ell$ and those primes at which~$E$ has bad reduction or at which $L/\Q$ ramifies. For $p\not \in S$, denote by $x_p$ the conjugacy class of $\frac{\varrho_{E,\ell}(\Frob_p)}{\sqrt p} \times \Frob_p$, where $\varrho_{E,\ell}$ is as in (\ref{equation: ladicrep}). The Haar measure~$\mu$ of~$G$ is the product of the Haar measure of $\SU(2)$ and the discrete measure on $\Gal(L/\Q)$.  We still denote by $\mu$ its image on the set $X$ of conjugacy classes of $G$.

Observe that the irreducible representations of $G$ are of the form $\Sym^m(\C^ 2)\otimes \varrho$, where $\C^ 2$ denotes the standard representation of $\SU(2)$, and $\varrho$ is an irreducible Artin representation of $\Gal(L/\Q)$.

\begin{theorem} One has:
\begin{enumerate}[i)]
\item If $m>0$ or $\varrho$ is nontrivial, the $L$-function $L(\Sym^m(\C^ 2)\otimes \varrho)$ is holomorphic and nonvanishing for $\Re(s)\geq 1$.
\item The sequence of $\{x_{p_i}\}_{i\geq 1}$ is $\mu$-equidistributed.
\item For any conjugacy class $c$ of $\Gal(L/\Q)$, the subsequence of  $\{a_{p_i}/\sqrt {p_i}\}_{i\geq 1}$ of those $p_i$ such that $\Frob_{p_i}=c$ is $\mu_{\rm{ST}}$-equidistributed.  
\end{enumerate}
\end{theorem}

\begin{proof} See \cite[Thm. 1]{MM}.
\end{proof}

Let $E$ be an elliptic curve with complex multiplication by a quadratic imaginary field $K$, and assume that $E$ is defined over $K$. Let $F=K$ and $L/K$ be a finite Galois extension. Let $G:=\Unitary(1)\times \Gal(L/K)$. Let~$S$ be the set of primes made up by $\ell$ and those primes at which~$E$ has bad reduction or at which $L/K$ ramifies. For $\p\not \in S$, let $x_\p$ be the conjugacy class defined by $\frac{\psi_{E}(\Frob_\p)}{\sqrt{N(\p)}} \times \Frob_\p$, where $\psi_E$ is as in (\ref{equation: heckedeuchar}). The Haar measure~$\mu$ of~$G$ is the product of the Haar measure of $\Unitary(1)$ and the discrete measure on $\Gal(L/K)$. Denote also by $\mu$ its image on the set $X$ of conjugacy classes of $G$.

Observe that the irreducible representations of $G$ are of the form $\phi_a\otimes \varrho$, where~$\phi_a$ is as in (\ref{equation: phia}), and~$\varrho$ is an irreducible Artin representation of $\Gal(L/k)$.

\begin{theorem} One has:
\begin{enumerate}[i)]
\item If $a\in\Z^*$ or $\varrho$ is nontrivial, the $L$-function $L(\phi_a\otimes \varrho)$ is holomorphic and nonvanishing for $\Re(s)\geq 1$.
\item The sequence $\{x_{\p_i}\}_{i\geq 1}$ is $\mu$-equidistributed.
\item For any conjugacy class $c$ of $\Gal(L/K)$, the subsequence of  $\{a_{\p_i}/\sqrt {N(\p_i)}\}_{i\geq 1}$ of those $\p_i$ such that $\Frob_{\p_i}=c$ is $\mu_{\rm{cm}}$-equidistributed.  
\end{enumerate}
\end{theorem}

\begin{proof}
One may reduce to the abelian case by following the same strategy as in \cite[Prop. 3.6]{FS12}.
\end{proof}

\subsubsection{A conjectural example}

Let $A$ be an abelian surface defined over the number field $F$. For a prime $\ell$, let
$$
\varrho_{A,\ell}\colon G_F\rightarrow \Aut(V_\ell(A))
$$
be the $\ell$-adic representation attached to $A$. Let $G:=\USp(4)$, $\mu$ its Haar measure, and $X$ its set of conjugacy classes. More explicitly, 
if $\mathfrak S_2$ denotes the symmetric group on two letters, there is a bijection between $[0,\pi]\times [0,\pi]/\mathfrak S_2$ and~$X$, obtained by sending the unordered pair $(\theta_1,\theta_2)$ to the conjugacy class of the diagonal matrix 
$$
\diag(e^{i\theta_1},e^{-i\theta_1},e^{i\theta_2},e^{-i\theta_2})\,.
$$ 
If $f\colon [0,\pi]\times [0,\pi]/\mathfrak S_2\rightarrow \C$ is a continuous function, then 
$$
\mu(f)=\frac{8}{\pi^2}\int_0^{\pi}\int_0^\pi f(\theta_1,\theta_2)(\cos(\theta_1)-\cos(\theta_2))^2\sin^2(\theta_1)\sin^2(\theta_2)d\theta_1 d\theta_2\,.
$$
Let $S$ the set of primes at which $A$ has bad reduction together with $\ell$, and for $\p\not\in S$, let $x_{\p}$ denote the conjugacy class of 
$$
\frac{\varrho_{A,\ell}(\Frob_\p)}{\sqrt {N(\p)}}\,.
$$
The fact that $x_{\p}$ defines a class of $X$ follows from the Weil Conjectures. The irreducible representations of $G$ are indexed by integers $a\geq b\geq 0$ and denoted by $\Gamma_{a,b}$ (see \cite[Chap. 16]{FH91}).
Via the Weyl Character Formula, they can be characterized in the following manner. Let $\{\alpha_1,\alpha_2, \overline \alpha_1, \overline \alpha_2\}$ be the eigenvalues of $x\in X$. If we denote by $J_d:=H_d(\alpha_1,\alpha_2, \overline \alpha_1, \overline \alpha_2)$, where $H_d$ stands for the $d$th complete symmetric polynomial in $4$ variables, then if $b\not=0$ (resp. $b=0$), we have
$$
\Tr(\Gamma_{a,b}(x))=\det\begin{pmatrix} J_a & J_{a+1}+J_{a-1}\\J_{b-1} & J_b+J_{b-2}\end{pmatrix} \qquad \left(\text{resp. } \Tr(\Gamma_{a,0}(x))=J_a \right)\,.
$$
The equidistribution of the classes $x_\p$ is linked to the holomorphicity and nonvanishing for $\Re(s)\geq 1$ of the $L$-functions $L(\Gamma_{a,b},s)$ (for $(a,b)\not=(0,0)$).

\begin{conjecture}\label{conjecture: g=2} If $\End_\Qbar(A)=\Z$, then the sequence $\{x_{\p_i}\}_{i \geq 1}$ is $\mu$-equidistributed.
\end{conjecture}

\section{The Sato-Tate group and the generalized Sato-Tate Conjecture}\label{section: def and conj}

In this section, given a smooth and projective variety $Y$ of dimension $n$ over a number field $F$ and a weight $0\leq \omega\leq 2n$, we define the Sato-Tate group attached to~$Y$ relative to the weight $\omega$. The reader might have found capricious the choices of the group $G$ in the examples of \S\ref{section: equiv and L}. The Sato-Tate group provides a uniform description of these choices of $G$. The description of the Sato-Tate group comes along with a general definition of the classes $x_\p$ that specializes to the \emph{ad hoc} definitions given in the examples, and once this is done we are ready to state the generalized Sato-Tate Conjecture. Both \S\ref{section: def ST} and \S\ref{section: stateconj} follow \cite[Chap.8]{Ser12}, where the general case of a scheme $Y$ of finite type over $\Z$ (not even separated) is considered. See \cite[Chap.9]{Ser12} and \cite{Ka13} for the relative case $Y\rightarrow T$, where~$T$ is an irreducible scheme of finite type over $\Z$ such that the residue field of its generic point has characteristic $0$. For the case of abelian varieties, we recall the relation between the Sato-Tate group and the Mumford-Tate group as described in \cite{BK12}. We end the section by considering the case of absolutely simple abelian varieties with complex multiplication.  

\subsection{Definition of the Sato-Tate group}\label{section: def ST}

Let $F$ be a number field and let~$Y$ be a smooth and projective variety defined over~$F$. Let~$S$ be a finite subset of the primes of~$F$ such that~$Y$ has good reduction\footnote{This means that there exists a smooth scheme over $\Spec(\mathcal O_F)\setminus S$, whose generic fiber is the $F$-scheme $Y$.} outside of $S$. 
Write $n:=\dim Y$, and choose an integer $0\leq \omega \leq 2n$. For any prime $\ell$, consider the $\omega$th \'etale cohomology group $H^{\omega}_{\et}(Y,\Q_\ell)$ of $\overline Y:=Y\times_F \overline F$. 
Let $m$ denote its dimension. The action of the absolute Galois group $G_F$ on $H^{\omega}_{\et}(Y,\Q_\ell)$ yields the $\ell$-adic representation 
$$
\varrho_Y^\omega\colon G_F\rightarrow \Aut(H^{\omega}_{\et}(Y,\Q_\ell))\subseteq \GL_m(\Q_\ell).
$$
For $\p\not \in S$, let $\Frob_\p^{-1}$ be a geometric Frobenius at $\p$, and define the local factor of~$Y$ at $\p$ to be the polynomial
$$
L_\p(Y,T):=\det(1-\varrho_Y^\omega(\Frob_\p^{-1})T;H^{\omega}_{\et}(Y,\Q_\ell))\,.
$$
The Weil Conjectures (see \cite[Th\'eor\`eme 1.6]{Del74}) state that $L_\p(Y,T)$ has integer coefficients, and that
$$
L_\p(Y,T)=\prod_{i=1}^{m}(1-\alpha_iT)\,,
$$
where the $\alpha_i$'s are $N(\p)$-Weil integers\footnote{This means that for every embedding $\sigma \colon \Q(\alpha_i)\hookrightarrow \C$, one has that $|\sigma(\alpha_i)|=N(\p)^{\omega/2}$.} of weight $\omega$.
The normalized local factor of~$Y$ at~$\p$ is defined as 
$$
\overline L_\p(Y,T):=L_\p(Y,N(\p)^{-\omega/2}T)\,.
$$

Let $G_\ell^{\omega}$ be the Zariski closure of $\varrho_Y^\omega(G_F)$, which we view as a $\Q_\ell$-algebraic subgroup of $\GL_m$. Let $G_\ell^{1,\omega}$ denote the Zariski closure of the image by $\varrho_Y^\omega$ of the kernel of the $\ell$-adic cyclotomic character $\chi_\ell$.
Choose an embedding 
$
\iota\colon \overline \Q_\ell\hookrightarrow \C
$ 
and denote by $G_{\ell,\iota}^{1,\omega}\subseteq \GL_m(\C)$ the group of $\C$-points of $G_\ell^{1,\omega}\times_\iota \C$.

\begin{definition}
The Sato-Tate group of $Y$ relative to the weight $\omega$ is a maximal compact subgroup of $G_{\ell,\iota}^{1,\omega}$. It is a compact real Lie group that we denote by $\ST(Y,\omega)$.
\end{definition}

Since all maximal compact subgroups of $G_{\ell,\iota}^{1,\omega}$ are conjugated, $\ST(Y,\omega)$ is well-defined up to conjugation. This will be enough for the purposes of equidistribution. Let $S_\ell$ be the union of $S$ and the set of primes of $F$ lying over $\ell$\footnote{Note a minor change of notation: the role played by the finite set $S$ in \S\ref{section: connL} will be played in this section by the finite set $S_\ell$.}.

\begin{definition} 
For $\p\not \in S_\ell$, let $x_\p$ denote the conjugacy class of $\ST(Y,\omega)$ corresponding to $\varrho_Y^\omega(\Frob_\p^{-1})^{ss}\otimes_\iota N(\p)^{-\omega/2}$, where the exponent ${}^{ss}$ means that we take the semisimple component.
\end{definition}

Observe that since $\varrho_Y^\omega(\Frob_\p^{-1})^{ss}\otimes_\iota N(\p)^{-\omega/2}$ has determinant $1$, it is indeed an element of $G_{\ell,\iota}^{1,\omega}$. Since it is semisimple, and all its eigenvalues have absolute value~$1$, it belongs to a compact subgroup. Since all compact subgroups are conjugated to $\ST(Y,\omega)$, it is conjugate to an element of $\ST(Y,\omega)$. And finally, it defines a conjugacy class of $\ST(Y,\omega)$, because the natural map $\Cl(\ST(Y,\omega))\rightarrow \Cl(G_{\ell,\iota}^{1,\omega})$ is injective (see \cite[\S8.3.3]{Ser12}).
Note that it is conjectured that $\varrho_Y^\omega(\Frob_\p^{-1})$ is semisimple (and it is known if, for example, $Y$ is an abelian variety).

By construction, we have that
$$
\det(1-x_{\p}T;H^{\omega}_{\et}(Y,\Q_\ell)\otimes_\iota\C)=\overline L_\p(Y,T).
$$ 

\subsection{Satement of the conjecture}\label{section: stateconj}

We can now state the main conjecture. Let $\{\p_i\}_{i\geq 1}$ be an ordering by norm of the primes of $F$ not in $S_\ell$.

\begin{conjecture}[generalized Sato-Tate]\label{conjecture: gen ST}
Let $X:=\Cl(\ST(Y,\omega))$ denote the set of conjugacy classes of $\ST(Y,\omega)$. Then:
\begin{enumerate}[i)]
\item The conjugacy class of $\ST(Y,\omega)$ in $\GL_m(\C)$ and the sequence $\{x_{\p_i}\}_{i\geq 1}\subseteq X$ depend neither on the choice of the prime $\ell$, nor on the choice of the embedding~$\iota$.
\item The sequence $\{x_{\p_i}\}_{i\geq 1}$ is equidistributed on $X$ with respect to the projection on this set of the Haar measure of $\ST(Y,\omega)$.
\end{enumerate}
\end{conjecture}

\subsection{Some examples revisited}

We now revisit some of the examples of~\S\ref{subsection: examples L}. In each case, we identify a smooth and projective variety $Y$, and show that the group~$G$ that we considered in~\S\ref{subsection: examples L} is precisely the Sato-Tate group of $Y$ (relative to the appropriate weight).
 
\subsubsection{Weight $\omega=0$}

Let $L/F$ be a finite Galois extension and let $Y$ be the $F$-scheme $\Spec(L)$. Let~$S$ be the set of primes of $F$ ramified in $L$.
Take $\omega=0$ and~$\ell$ a prime not in $S$. Since, in this case, \'etale cohomology is identified with Galois cohomology, we have an isomorphism of $G_F$-modules
$$
H^0_{\et}(\Spec(L),\Q_\ell)\simeq H^0(G_L,\overline L)\otimes \Q_\ell=L\otimes\Q_\ell\,.\footnote{Recall that the $0$th Galois cohomology group $H^0(G,M)$ of a $G$-module $M$ is the set of invariant elements $M^G$.}
$$ 
Therefore we have that $m=[L:F]$. Moreover, by taking a normal basis of $L/F$, we see that $\varrho_Y^0$ is identified with the permutation  representation of $\Gal(L/F)$ on the space $L\otimes\Q_\ell$. Since $\ell\not \in S$, the image of $\varrho_Y^0$ coincides with the image of the kernel of the cyclotomic character. Clearly, this image is Zariski closed, and isomorphic to $\Gal(L/F)$. Extending scalars to $\C$ and taking a maximal compact subgroup does not change anything in this case, so we obtain that $\ST(\Spec(L),0)$ is identified with $\Gal(L/F)$.

The Haar measure of $\Gal(L/F)$ is the discrete measure.
For $\p \not \in S_\ell$, the class~$x_{\p}$ is identified with $\Frob_\p^ {-1}\in\Cl(\Gal(L/F))$\footnote{Comparing with \S\ref{section: Chebotarev}, the reader may find the discrepancy of a ${}^{-1}$ exponent, introduced by the choice of a geometric Frobenius. However, the corresponding equidistribution statements are obviously equivalent.}. In this case, the statement of Conjecture \ref{conjecture: gen ST} is equivalent to the statement of the Cebotarev Density Theorem (see Corollary \ref{corollary: Cebotarev}).

\subsubsection{Weight $\omega=1$}

Let $Y=A/F$ be an abelian variety of dimension $g$. Then the following well-known isomorphism 
$$
H^\omega_{\et}(A,\Q_\ell)\simeq \bigwedge^{\omega} H^1_{\et}(A,\Q_\ell)\,,
$$
leads to
$$
\ST(A,\omega)\simeq \bigwedge^{\omega}\ST(A,1).\footnote{Note that the exterior power of $\ST(A,1)$ makes sense, since $\ST(A,1)$ is defined as a representation of a Lie group, rather than just as a Lie group.}
$$
Therefore, for abelian varieties we will take $\omega=1$, and we will only consider $\ST(A):=\ST(A,1)$. Let $S$ be the set of primes of bad reduction of $A$.

\textbf{The case $g=1$ without CM.}
Let $A=E$ be an elliptic curve defined over the number field $F$ and assume that is does not have complex multiplication. By Serre's open image Theorem (see \cite{Ser72}),  $G_\ell^1$ (resp. $G_\ell^{1,1}$) is $\GSp_2/\Q_\ell$ (resp. $\Sp_2/\Q_\ell$). 

A maximal compact subgroup of  $\Sp_2(\Q_\ell)\otimes_\iota \C$ is $\USp(2)=\SU(2)$. Thus we have that
$\ST(E)=\SU(2)$, and we see that in this concrete case the statement of Conjecture \ref{conjecture: gen ST} is equivalent to the statement of Corollary \ref{corollary: equid no cm}.

\textbf{The case $g=1$ with CM.}
Let $A=E$ be an elliptic curve with complex multiplication by $K$ and assume that $K\subseteq F$. Then the Zariski closure $G_\ell^1$
 is $T^K(\Q_\ell)$, where $T^K=\Res_{K/\Q}\mathbb G_m$ is the two dimensional torus\footnote{Recall that for any commutative $\Q$-algebra $A$, we have $T^K(A)=(A\otimes_\Q K)^*$.} defined by~$K$. Thus, extending scalars to $\C$, we have
$$
T^K(\Q_\ell)\otimes_{\iota}\C=(\Q_\ell\otimes_\Q K)^*\otimes_{\iota}\C=\C^*\times \C^*.
$$ 
Thus
$$
G^{1,1}_{\ell,\iota}=\left\{\left(z,\frac{1}{z}\right)\, |\, z\in \C^*\right\}\subseteq T^K(\C)\,.
$$
The maximal compact subgroup of $G^{1,1}_{\ell,\iota}\simeq \C ^*$ is $\Unitary (1)$, and thus $\ST(E)=\Unitary(1)$. We see that in this particular case the statement of Conjecture \ref{conjecture: gen ST} is equivalent to the statement of Corollary \ref{corollary: STCM}. If $F$ does not contain $K$, then $\ST(E)$ is the normalizer of the embedded $\Unitary(1)$ inside $\SU(2)$.

\textbf{Abelian variety $A$ of dimension $g$ with $\End_\Qbar(A)=\Z$.}
If $g$ is odd (or if $g=2$ or $6$) and $\End_\Qbar(A)=\Z$, a result of Serre shows that 
$G_\ell^{1,1}$ is $\Sp_{2g}/\Q_\ell$. Thus $\ST(A)=\USp(2g)$, and for $g=2$ we have that Conjecture ~\ref{conjecture: gen ST} reduces to Conjecture~\ref{conjecture: g=2}.

\subsection{The connection with the Mumford-Tate group}\label{section: MT group}

In this section, let $Y=A$ be an abelian variety of dimension $n=g$ defined over the number field~$F$. For a field extension $L/F$, we will denote by $A_L$ the base change of~$A$ from~$F$ to~$L$. Fix an embedding $F\hookrightarrow \C$.
The $2g$-dimensional real vector space $H_1(A_\C,\R)$ is endowed with a complex structure
$$
h\colon \C \rightarrow \End_\R(H_1(A_\C,\R))\,,
$$
obtained by identifying $H_1(A_\C,\R)$ with the tangent space of $A$ at the identity.
\begin{definition}

\begin{enumerate}[i)]
\item The Mumford-Tate group $\MT(A)$ of $A$ is the smallest algebraic subgroup of $\GL(H_1(A_\C,\Q))$ over $\Q$ such that $G(\R)$ contains $h(\C^*)$.
\item The Hodge group $\Hg(A)$ of $A$ is defined as $\MT(A)\cap \Sp_{2g}$.
\end{enumerate}
\end{definition}

Deligne \cite[I, Prop. 6.2]{Del82} showed that for every prime $\ell$ the connected component of $G_{\ell}^{1}$ (resp. of $G_{\ell}^{1,1}$) is contained in $\MT(A)\times_\Q\Q_\ell$ (resp. $\Hg(A)\times_\Q\Q_\ell$). Let us denote by $G^0$ the connected component of a topological group $G$.

\begin{conjecture}[Mumford-Tate] The inclusion $(G_{\ell}^{1})^0\subseteq \MT(A)\times_\Q\Q_\ell$ is an equality. Or equivalently, the inclusion $(G_{\ell}^{1,1})^0\subseteq \Hg(A)\times_\Q\Q_\ell$ is an equality.
\end{conjecture}

It follows that the Mumford-Tate group only accounts for the connected component of the Sato-Tate group.
The next conjecture (see \cite[C.3.3]{Ser77} and \cite{BK11}), which predicts the existence of an algebraic group that should account for the whole group $G_\ell^{1,1}$, may be seen as a refinement of the Mumford-Tate Conjecture.

\begin{conjecture}[Algebraic Sato-Tate]\label{conjecture: AST} There exists an algebraic subgroup $\AST(A)$ of $\Sp_{2g}/\Q$, called the \emph{Algebraic Sato-Tate group of $A$}, such that for each prime $\ell$, $G_\ell^{1,1}=\AST(A)\times_\Q\Q_\ell$.
\end{conjecture}

When the Algebraic Sato-Tate Conjecture and Mumford-Tate Conjecture are true, then one can obtain the Sato-Tate group $\ST(A)$ (resp. the identity connected component $\ST(A)^0$) as a maximal compact subgroup of the group of complex points of $\AST(A)\times_\Q \C$ (resp. $\Hg(A)\times_\Q \C$). In particular, in this case, the Sato-Tate group depends neither on the choice of a prime $\ell$, nor on the choice of an embedding $\iota\colon \Q_\ell\hookrightarrow \C$ (we will come back to the Algebraic Sato-Tate Conjecture in \S\ref{section: clasres}).

\subsection{Yet another example: Abelian varieties with complex multiplication}\label{section: STCM}

Let $A$ be an abelian variety of dimension $g$ defined over the number field $F$. Suppose that $A$ has complex multiplication over $F$ in the sense that $\End(A_F)\otimes\Q$ contains a number field $E$ of degree $2g$. We will follow the exposition of \cite{Rib80} for Kubota's definition of the rank of a CM-type.

Fix an embedding of $\Qbar$ into $\C$. Suppose that $F/\Q$ is a Galois extension and that $E\subseteq F\subseteq\Qbar$, and let 
$$
G:=\Gal(F/\Q),\qquad H:=\Gal(F/E)\,.
$$
We will use the convention that $G$ acts on $E$ on the right. We identify $\Hom_\Q(E,\C)$ with $H\backslash G$, and write $c$ for complex conjugation\footnote{In \cite[\S4]{FGL13}, the case $F=E=K$ and $G=(\Z/\ell Z)^*$ was considered. This section is a straightforward generalization of that work.}.
The data $(A/F,E)$ define a CM-type $\mathcal S\subseteq \Hom_\Q(E,\C)$, that is, a subset of $H\backslash G$ such that $H\backslash G$ is the disjoint union of $\mathcal S$ and $\mathcal S c$. Denote by $\tilde {\mathcal S}$ the set $\{g\in G\,|\, Hg \in \mathcal S\}$. If we let $\tilde {\mathcal R}$ be the set $\tilde {\mathcal S}^{-1}$ of inverse elements of $\tilde {\mathcal S}$, define
$$
H':=\{ g\in G\,|\, g\tilde {\mathcal R} = \tilde {\mathcal R}\}\,.
$$ 
The subfield $K$ of $F$ fixed by $H'$ is called the \emph{reflex} field of $E$. It is known to be a CM field contained in $F$. Moreover, the image $\mathcal R$ of $\tilde {\mathcal R}$ in $H'\backslash G=\Hom(K,\C)$ defines a CM-type of $K$. 

Given an algebraic torus $T$, we will denote by $X(T)$ the right $G_\Q$-module 
$$
\Hom(T_\Qbar,\mathbb G_{m,\Qbar})\,,
$$
called the character group of $T$.
For a number field $L$, let $T^L$ denote the torus $\Res_{L/\Q}(\mathbb{G}_m)$. Observe that $X(T^L)$ may be identified with $\bigoplus_{\sigma\in\Hom_\Q(L,\C)}\Z[\sigma]$. It is equivalent to give a homomorphism $T^K\rightarrow T^E$ than to give a homomorphism $X(T^E)\rightarrow X(T^K)$. Let 
$$
\Phi\colon T^K\rightarrow T^E
$$
be the map, whose pull-back on character groups is given by
$$
\Phi^*\colon X(T^E)\rightarrow X(T^K)\,,\qquad \Phi^*([\sigma])=\sum_{r\in \mathcal R}[\tilde r\tilde \sigma]\,,\qquad \text{for }\sigma\in  \Hom_\Q(E,\C)\,.
$$
Here,~$\tilde\sigma$ and~$\tilde r$ denote prolongations to~$G$ of~$\sigma$ and~$r$, and for $g\in G$ we denote by~$[g]$ the induced element in $\Hom_\Q(K,\C)$. Note that $\Phi^*$ is independent of the choices of ~$\tilde\sigma$ and~$\tilde r$. Let $T_0$ denote the image of $\Phi$. The \emph{rank} of the CM-type $\mathcal S$ is defined as the dimension of $T_0$, or equivalently, as the rank of the $\Z$-submodule $\Phi^*(X(T^E))$ of $X(T^K)$. 
Consider the matrix
$$
D:=(i(\sigma,\tau))_{\sigma\in \mathcal R,\tau\in \mathcal S}\,,\quad\text{ where }\quad
i(\sigma,\tau)=
\begin{cases}
1 & \text{if $\tilde\sigma\tilde\tau^{-1}\in \tilde {\mathcal R}$,}\\
0 & \text{if $\tilde\sigma\tilde\tau^{-1}\not\in \tilde {\mathcal R}$.}
\end{cases}
$$
Here $\tilde\sigma,\tilde\tau$ are prolongations to $G$ of $\sigma,\tau$. Let $\nu$ denote the rank of the matrix $D$.

\begin{lemma}[Lemma 1, \cite{Kub65}]\label{lemma: image phik} The rank of $\mathcal S$ is equal to $\nu+1$.
\end{lemma}

\begin{proof}
Let $\overline D=(i(\sigma,\tau))_{\sigma\in H'\backslash G,\tau\in H\backslash G}$ denote the matrix of $\Phi^*$ in the basis $\mathcal S\cup \mathcal Sc$ for $H\backslash G$, and $\mathcal R\cup \mathcal Rc$ for $H'\backslash G$. Write $U$ for the $k\times g$ matrix whose entries are all ones, where $g$ and $k$ are integers such that $[E:\Q]=2g$ and $[K:\Q]=2k$.
We obtain
$$
\overline D=
\begin{pmatrix}
D & U-D\\
U-D & D
\end{pmatrix}\sim
\begin{pmatrix}
D & U-D\\
U  & U
\end{pmatrix}\sim
\begin{pmatrix}
D & U \\
U  & 2U
\end{pmatrix}\sim
\begin{pmatrix}
D & 0 \\
U  & U
\end{pmatrix}\sim
\begin{pmatrix}
D & 0 \\
0  & U
\end{pmatrix}\,,
$$
where for the penultimate equivalence we have used that the column corresponding to $\tau=1$ has only ones. 
\end{proof}

\begin{proposition} Let $A$ be an absolutely simple abelian variety defined over a number field $F$ such that $F/\Q$ is Galois. Suppose that $A$ has CM by a field $E\subseteq F$. Let $\nu+1$ be the rank of the CM-type of $A$. Then 
the Sato-Tate group of $A$ is isomorphic to
$\Unitary(1)\times\stackrel{\nu}\ldots\times \Unitary(1)\,.$
\end{proposition}

\begin{proof} As in \S\ref{section: def ST}, let $G_\ell^1$ denote the Zariski closure of the image of the $\ell$-adic representation attached to $A$. On the one hand, the hypothesis that $A$ is absolutely simple and that $E\subseteq F$ imply that $G_\ell^1$ is connected. On the other hand, since the Mumford-Tate Conjecture is known to hold for CM abelian varieties (see for example \cite{Yu13}), we have that (the connected component of) $G_\ell^1$ coincides with $\MT(A)\times \Q_\ell$ for every prime $\ell$. As explained in \S\ref{section: MT group}, we may then obtain $\ST(A)$ as a maximal compact subgroup of the group of complex points of $\Hg(A)\times_\Q\C$. By \cite[Example 3.7]{Del82}, we have that $\MT(A)$ is $T_0$, and therefore 
$\MT(A)(\C)\simeq \C^*\times\stackrel{\nu+1}\ldots\times \C^*$. Thus, $\Hg(A)(\C)\simeq \C^*\times\stackrel{\nu}\ldots\times \C^*$ from which the proposition follows. 
\end{proof}

We now recall a fundamental result of the theory of Shimura and Taniyama for abelian varieties with complex multiplication. 
\cite[Theorem 19.11]{Shi98} establishes that there exists an algebraic Hecke charakter $\chi$ of $F$ with values in $E$ and infinity type $\tilde {\mathcal R}$ such that for every prime $\p$ at which $A$ has good reduction
$$
L_\p(A,T)=\prod_{\tau\in\Hom(E,\C)}(1-\chi^\tau(\p)T)
$$
(see for example \cite[Theorem 9.1]{MT11}). Here we are making use of a slightly more general notion of algebraic Hecke character than the one used in (\ref{equation: Hecke char}). For an integral ideal $\mathfrak M$ of $F$, an algebraic Hecke character of $F$ with values in $E$ and infinity type $\tilde {\mathcal R}$ is a group homomorphism
$$
\chi\colon I_\mathfrak M\rightarrow E^*
$$
such that
\begin{equation}\label{equation: grossen}
\chi(\alpha \mathcal O_F)=\prod_{\sigma\in \tilde {\mathcal R}}\alpha^{\sigma}\qquad\text{for every $\alpha\in F^*$ with}\qquad \alpha\equiv^\times 1\pmod {\mathfrak M}\,.
\end{equation}
Here $I_\mathfrak M$ denotes the group of fractional ideals of $F$ coprime to $\mathfrak M$.

\begin{theorem}
Let $A$ be an absolutely simple abelian variety defined over a number field $F$ such that $F/\Q$ is Galois. Suppose that $A$ has CM by a field $E\subseteq F$. Then the generalized Sato-Tate Conjecture holds for $A$. More explicitly, let $\{\p_i\}_{i\geq 1}$ be an ordering by norm of the primes at which $A$ has good reduction. Let $\{\tau_1,\dots,\tau_\nu\}$ be a subset of $\mathcal S\subseteq H\backslash G=\Hom_\Q(E,\C)$ such that $(i(\sigma,\tau_j))_{\sigma\in \mathcal R,j=1,\dots,\nu}$ has rank $\nu$. Then the sequence 
$$
\left\{\left(\frac{\chi^{\tau_1}(\p_i)}{\sqrt{N(\p_i)}},\dots,\frac{\chi^{\tau_\nu}(\p_i)}{\sqrt{N (\p_i)}}\right)\right\}_{i\geq 1}\subseteq \Unitary(1)\times\stackrel{\nu}\ldots\times\Unitary(1)\simeq \ST(A)
$$
is equidistributed over $\Unitary(1)\times\stackrel{\nu}\ldots\times\Unitary(1)$ with respect to the Haar measure.
\end{theorem}

\begin{proof}
The irreducible representations of $\Unitary(1)\times\stackrel{\nu}\ldots\times\Unitary(1)$ are 
the characters
\begin{equation}\label{equation: irred rep}
\phi_{b_1,\dots,b_{\nu}}\colon \Unitary(1)\times \stackrel{\nu}\ldots\times\Unitary(1)\rightarrow \C^*,\qquad \phi_{b_1,\dots,b_{\nu}}(z_1,\dots,z_{\nu})=\prod_{i=1}^{\nu}z_i^{b_i}\,,
\end{equation}
where $b_1,\dots,b_{\nu}\in \Z$. By Theorem \ref{theorem: equidistribution}, it suffices to prove that for any $b_1,\dots,b_{\nu}\in \Z$, not all of them zero, the $L$-function
$$
\prod_{i\geq 1}\left(1-\frac{\prod_{j=1}^\nu \chi^{\tau_j}(\p_i)^{b_{j}}}{N(\p_i)^{\sum_{j=1}^{\nu}b_j/2}}N(\p_i)^{-s  }\right)^{-1}
$$
is holomorphic and nonvanishing for $\Re(s)\geq 1$.
But, up to a finite number of local Euler factors, this is just the $L$-function of the unitarized algebraic Hecke charakter 
$$
\Psi:=\frac{\prod_{j=1}^\nu \chi^{\tau_j}(\cdot)^{b_{j}}}{N(\cdot)^{\sum_{j=1}^{\nu}b_j/2}}\,.  
$$
By `unitarized' we mean that it takes values in $\Unitary(1)\subseteq \C^*$. By Theorem \ref{theorem: Hecke}, the $L$-function of a nontrivial unitarized algebraic Hecke charakter is holomorphic and nonvanishing for $\Re(s)\geq 1$.  Therefore, it only remains to prove that the Hecke charakter $\Psi$ is nontrivial. Suppose that $\Psi$ were trivial, and let $B:=\frac{b_1+\dots+b_{\nu}}{2}$. Then for every $\alpha\in K^*$ such that $\alpha\equiv^\times 1 \pmod{\mathfrak M}$ equation (\ref{equation: grossen}) implies
\begin{equation}\label{equation: inf dem}
\begin{array}{lll}
1 =\Psi(\alpha\mathcal O_k)&=&\displaystyle{\alpha^{\sum_{j=1}^\nu b_j\sum_{\sigma\in \tilde R} [\sigma\tilde\tau_j] -\sum_{\sigma\in G}B[\sigma]}}\\[4pt]
& = & \displaystyle{\alpha^{[F:K]\left(\sum_{j=1}^\nu b_j\Phi^*([\tau_j]) -\sum_{\sigma\in H'\backslash G}B[\sigma]\right)}} \\[4pt]
&=&\displaystyle{\alpha^{[F:K]\left(\sum_{\sigma\in H'\backslash G} \sum_{j=1}^\nu b_ji(\sigma,\tau_j)[\sigma] -\sum_{\sigma\in H'\backslash G}B[\sigma]\right)}}\,.
\end{array}
\end{equation}
A well-known theorem of Artin states that the characters of a monoid in a field are linearly independent. Viewing the elements $\sigma\in H'\backslash G=\Hom(K,\C)$ as characters of the monoid $\{\alpha\in K^*\,|\,\alpha\equiv^\times 1\pmod{\mathfrak M}\}$ in $\C$, we deduce that for every $\sigma\in H'\backslash G$, we have $\sum_{j=1}^\nu b_j i(\sigma,\tau_j)=B$. In particular, for $\sigma=1$, we have
$$
B=\sum_{j=1}^\nu b_j i(1,\tau_j)=\sum_{j=1}^\nu b_j=2B\,,
$$ 
from which we deduce that $B=0$. This implies that
$$
\sum_{j=1}^\nu b_j i(\sigma,\tau_j)=0\qquad \text{for every $\sigma \in \mathcal R$,}
$$
which is a contradiction with the fact that $(i(\sigma,\tau_j))_{\sigma\in \mathcal R,j=1,\dots,\nu}$ has rank $\nu$, and that not all of the $b_j$'s are zero.
\end{proof}

See \cite[Prop. 15]{Joh13} for a more general result valid for abelian varieties with potential complex multiplication over a number field, that is, abelian varieties over a number field adquiring complex multiplication after base change to a finite algebraic extension.

\section{Group-theoretic classification results}\label{section: GT results}

In this section we recall some results of \cite{FKRS12} and \cite{FKS12} concerning the classification of Sato-Tate groups of certain families of selfdual motives with rational coefficients and fixed values of weight and Hodge numbers. Serre \cite{Ser95} has given a more general construction of the Sato-Tate group (than the one seen in \S\ref{section: def ST}) that applies to the context of motives. Moreover, Serre has listed a series of properties that the Sato-Tate group should enjoy, assuming the validity of certain conjectures from the motivic folklore (e.g. the Mumford-Tate Conjecture). These are the so-called \emph{Sato-Tate axioms}. 

It is in this context that the problem of obtaining a group-theoretic classification of Sato-Tate groups for
selfdual motives with rational coefficients of fixed weight $\omega$, dimension $m$, and Hodge numbers $h^{p,q}$ arises. 

More precisely, this problem consists on identifying groups
obeying the following axioms (as formulated in \cite{FKRS12}):
\begin{enumerate}
\setlength\itemindent{10pt}
\item[(ST1)]
The group $G$ is a closed subgroup of~$\USp(m)$ (if~$\omega$ is odd) or $\operatorname{O}(m)$ (if~$\omega$ is even).
\item[(ST2)] (Hodge condition)
There exists a subgroup $H$ of $G$, called a \emph{Hodge circle}, which is the image of 
a homomorphism $\theta: \Unitary(1) \to G^0$
such that $\theta(u)$ has eigenvalues $u^{p-q}$ with multiplicity $h^{p,q}$.
Moreover, $H$ can be chosen so that the conjugates of $H$ generate a dense subgroup of $G^0$. Here, $G^0$ stands for the connected component of $G$.
\item[(ST3)] (Rationality condition)
For each component $C$ of $G$ and each irreducible character $\chi$ of $\GL_{m}(\C)$, the expected value 
$$
\int_{g\in C}\chi(g) \mu(g)
$$
of $\chi(\gamma)$ over $\gamma\in C$ under the Haar measure of $G$ is an integer.
\end{enumerate}

\begin{remark} Consider the tautological representation of $\GL_m(\C)$ on a complex vector space $V\simeq \C^ m$ of dimension $m$. To test $\mathrm{(ST3)}$, we often take the representation $V\otimes V^*$, where $V^*$ denotes the contragredient representation of $V$. Let $\chi$ denote the character of $V\otimes V^*$, and let $\mu$ denote the Haar measure of $G$. In this case
$$
\int_{g\in G}\chi(g) \mu(g)=\int_{g\in G} \Tr(g)\Tr(\overline g^t) \mu(g)=\int_{g\in G} |\Tr(g)|^2 \mu(g)=\dim(\Hom_{\C[G]}(V,V))
$$
is always an integer. It is the \emph{finer} condition
$$
\int_{g\in C} |\Tr(g)|^2 \mu(g)\in \Z\,,
$$
where $C$ is a connected component of $G$, that is actually a restrictive condition.
\end{remark}

\begin{remark}\label{remark: not rat}
Observe that $\mathrm{(ST1)}$ and $\mathrm{(ST3)}$ do not necessarily hold if one considers motives that are not selfdual or without rational coefficients. For example, let $f\in S_k(\Gamma_0(M),\varepsilon)$ be a newform without complex multiplication. Note that in this case one has
$$
\omega=k-1, \qquad m=2, \qquad h^{k-1,0}=h^{0,k-1}=1\,.
$$ 
Then the Sato-Tate group of the motive associated to $f$ is given by (\ref{equation: STgmf}). Its components do not necessarily have rational moments, and it is obviously not contained in $\SU(2)=\USp(2)$.
\end{remark}

\begin{remark}
For fixed $\omega,m,h^{p,q}$, there are only finitely many groups $G$ satisfying $\mathrm{(ST1)}$, $\mathrm{(ST2)}$, $\mathrm{(ST3)}$
up to conjugation within $\USp(m)$ or $\operatorname{O}(m)$ (see \cite[Remark~3.3]{FKRS12}). Note however, that if we ignore the hypothesis of selfduality or of rationality of the coefficients, then for fixed $\omega,m,h^{p,q}$ there may exist infinitely many possibilities for the Sato-Tate group. Take for example $f_n\in S_k(\Gamma_0(M_n),\varepsilon_n)$ such that $\ord(\varepsilon_n)$ tends to $\infty$ with $n$. 
\end{remark}

For the case of weight $\omega=1$ and Hodge numbers $h^{0,1}=h^{1,0}=1$, there are up to conjugation only 3 different groups satisfying $\mathrm{(ST1)}$, $\mathrm{(ST2)}$, $\mathrm{(ST3)}$: 
\begin{equation}\label{equation: w=1}
\Unitary(1),\medspace N_{\SU(2)}(\Unitary(1)),\medspace \SU(2)\,.
\end{equation}
Here, by $N_{\SU(2)}(\Unitary(1))$, we mean the normalizer of 
$$
\left\{\begin{pmatrix}
z & 0\\
0 & \overline z
\end{pmatrix} \, |\, z\in \C^*,\,|z|=1\right\}
$$
inside $\SU(2)$. The classification has been carried out in two other nontrivial cases.

\begin{theorem}\cite[Thm. 1.2.]{FKRS12}\label{theorem: classification w=1} If $\mathrm{(ST1)}$, $\mathrm{(ST2)}$, $\mathrm{(ST3)}$ are true, then the Sato-Tate group of a selfdual motive with rational coefficients of weight 1 and Hodge numbers $h^{0,1}=h^{1,0}=2$ is conjugate to one of $55$ particular groups. The possible connected components of these groups are:
$$
\Unitary(1),\medspace \SU(2),\medspace \Unitary(1) \times \Unitary(1),\medspace \Unitary(1) \times \SU(2),\medspace \SU(2) \times \SU(2),\medspace \USp(4)\,.
$$
The number of groups (up to conjugation) with each connected component is $32$, $10$, $8$, $2$, $2$, $1$, respectively.
\end{theorem}

The weight 1 classification of Theorem \ref{theorem: classification w=1} turns out to be very similar to the following weight 3 classification (which, in fact, turns out to be easier).

\begin{theorem}[\cite{FKS12}] If $\mathrm{(ST1)}$, $\mathrm{(ST2)}$, $\mathrm{(ST3)}$ are true, then the Sato-Tate group of a selfdual motive with rational coefficients of weight 3 and Hodge numbers $h^{3,0}=h^{2,1}=h^{1,2}=h^{0,3}=1$ is conjugate to one of $26$ particular groups. The possible connected components of these groups are:
$$
\Unitary(1),\medspace \SU(2),\medspace \Unitary(1) \times \Unitary(1),\medspace \Unitary(1) \times \SU(2),\medspace \Unitary(2),\medspace \SU(2) \times \SU(2),\medspace \USp(4)\,.
$$
The number of groups (up to conjugation) with each connected component is $10$, $1$, $8$, $2$, $2$, $2$, $1$, respectively.
\end{theorem}

\subsection{Classification results for abelian varieties with $g\leq 3$}\label{section: clasres}

For $A$ an abelian variety defined over a number field $F$ of dimension $g\leq 3$, the Sato-Tate axioms are known to hold. This is a consequence of the fact that strong results relating the arithmetic of $A$ and the Sato-Tate group are known, as we describe below.

\begin{theorem}\cite[Prop. 3.2]{FKRS12}\label{theorem: truth ST axioms} For the case of an abelian variety $A$ defined over a number field $F$ of dimension $g$ (weight $\omega=1$ and Hodge numbers $h^{0,1}=h^{1,0}=g$) the Sato-Tate axioms $\mathrm{(ST1)}$, $\mathrm {(ST2)}$, $\mathrm{(ST3)}$ are true if $A$ satisfies the Algebraic Sato-Tate Conjecture and the Mumford-Tate Conjecture (see \S\ref{section: MT group}). 
\end{theorem}

\begin{proof}[Sketch of proof]
Loosely speaking, one should see (ST1) as a consequence of the Weil Pairing, (ST2) as a consequence of the Mumford-Tate Conjecture and Hodge Theory, and (ST3) as a consequence of the algebraicity of the Sato-Tate group.
\end{proof}

\begin{theorem}\cite[Thm. 6.1, Thm 6.10.]{BK12}\label{theorem: truth MTAST} For an abelian variety $A$ of dimension $g\leq 3$ defined over a number field $F$, the Algebraic Sato-Tate Conjecture and the Mumford-Tate Conjecture are true. 
\end{theorem}

\begin{proof}[Sketch of proof] Consider the \emph{Twisted Lefschetz group} of $A$, which is defined as the union
$$
\TL(A):=\bigcup_{\tau\in G_F} \Lef(A,\tau)\,,
$$
where
$$
\Lef(A,\tau):=\{\gamma\in\Sp_{2g}\,|\, \gamma^{-1}\alpha\gamma={}^{\tau}\alpha \text{ for all } \alpha\in \End(A_\Qbar) \}\,.
$$
Here one should see $\alpha$ as an element of $\End(H_1(A_\C,\Q))$.
There is an obvious inclusion of $G_{\ell}^{1,1}$ in $\TL(A)\times_\Q\Q_\ell$, but in general this inclusion can be strict, as it happens with the famous Mumford examples\footnote{Indeed, for such examples $G_{\ell}^{1,1}\subsetneq \Sp_{2g}/\Q_\ell$ and $\End_\Qbar(A)=\Z$, from which it follows that $\TL(A)=\Sp_{2g}/\Q$.} in dimension $g=4$ (see \cite{Mum69}).
However, for $g\leq 3$, building on much existing literature on Mumford-Tate groups, Banaszak and Kedlaya can prove that $G_{\ell}^{1,1}= \TL(A)\times_\Q\Q_\ell$, and thus Conjecture~\ref{conjecture: AST} holds with $\AST(A)=\TL(A)$.
\end{proof}

As a consequence, for an abelian variety of dimension $g\leq 3$, the Sato-Tate axioms (ST1), (ST2), (ST3) are true. In our attemp to reach a classification of Sato-Tate groups of abelian varieties of dimension $g\leq 3$, we want to further investigate the arithmetic of the abelian variety in terms of the Sato-Tate group. For this, we introduce the  notion of \emph{Galois endomorphism type} of an abelian variety.

\begin{definition}\label{definiton: Galois type}
Consider pairs $[G,E]$ in which $G$ is a finite group and $E$ is a finite-dimensional
$\R$-algebra equipped with an action of $G$ by $\R$-algebra automorphisms.
An \emph{isomorphism} between two such pairs $[G,E]$ and $[G', E']$ consists of an isomorphism
$G \simeq G'$ of groups and an equivariant isomorphism $E \simeq E'$ of $\R$-algebras. For an abelian variety $A$ defined over $F$, let $K/F$ denote the minimal extension over which all the endomorphisms of $A$ are defined.
The \emph{Galois endomorphism type} associated to~$A$ is the isomorphism class of the pair
$[\Gal(K/F),\End(A_K)\otimes{\R}]$.
Note that abelian varieties defined over different number fields may have the same Galois endomorphism type.
\end{definition}

The next result will be crucial for our classification.

\begin{theorem}\cite[Prop. 2.19]{FKRS12}\label{theorem: ST->GT} For an abelian variety $A$ of dimension $g\leq 3$ defined over a number field $F$, the Sato-Tate group uniquely determines the Galois endomorphism type.
\end{theorem}

\begin{proof}[Sketch of proof]
The proof is built on the following well-known property of the Hodge group 
$$
\End(A_\C)\otimes \Q= \End(H_1(A_\C,\Q))^{\Hg(A)}
$$
(see \cite{Rib83}, for example). By the definition of $K$, and tensoring with $\C$, we obtain
\begin{equation}\label{equation: hodge}
\End(A_K)\otimes\C = (\End(H_1(A_\C,\Q))\otimes\C)^{\ST(A)^0}\,,
\end{equation}
since $\ST(A)^0$ is a maximal compact subgroup of $\Hg(A)$. The proof of Theorem~\ref{theorem: truth MTAST} yields an isomorphism $\Gal(K/F)\simeq \ST(A)/\ST(A)^0$. We can thus recover the action of $\Gal(K/F)$ on the left hand side of (\ref{equation: hodge}) by means of the action of $\ST(A)$ on the right hand side of (\ref{equation: hodge}). To recover $\End(A_\C)\otimes \R$, note that this is the unique $\R$-subspace of $\End(A_K)\otimes \C$ of half the dimension over which the real part of the Rosati form is positive definite\footnote{Let $\phi\colon A\rightarrow \hat A$ be a polarization on $A$. The Rosati form is defined in the following way
$$
f_R\colon \End(A_K)\otimes\C \rightarrow \C\,,\qquad f_R(\psi)=\Tr(\psi\circ\psi'; H_1(A_\C,\Q))\,,
$$
where $'$ denotes the Rosati involution, that is, $\psi':=\phi^{-1}\circ \hat \psi \circ \phi$.}.
\end{proof}

\begin{theorem}[g=1]\label{theorem: STgroupsg1} The Sato-Tate group and the Galois endomorphism type of an elliptic curve~$E$ defined over a number field~$F$ uniquely determine each other. They are restricted to a list of 3 possibilities, each of which occurs for a particular choice of~$E$ and~$F$.
\end{theorem}

\begin{proof} One implication is given by Theorem \ref{theorem: ST->GT}. Moreover, the proof of that theorem is effective, i.e. from the Sato-Tate group one can compute the Galois endomorphism type. The 3 Sato-Tate groups 
$$
\Unitary(1),\qquad N_{\SU(2)}(\Unitary(1)),\qquad \SU(2)
$$
listed in (\ref{equation: w=1}) give rise to the 3 following Galois endomorphism types
$$
[C_1,\C],\qquad [C_2,\C], \qquad [C_1,\R]\,.
$$
Here, $C_1$ denotes the trivial group, and $C_2$ is the group of 2 elements. 
These Galois endomorphism types respectively correspond to an elliptic curve defined over $F$ with CM defined over $F$, to an elliptic curve with CM but not defined over $F$, and to an elliptic curve without complex multiplication.
\end{proof}

\begin{theorem}[g=2]\label{theorem: STgroupsg2} \cite[Thm. 4.3]{FKRS12} The Sato-Tate group and the Galois endomorphism type of an abelian surface~$A$ defined over a number field $F$ uniquely determine each other. They are restricted to a list of 52 possibilities, each of which occurs for a particular choice of $A$ and $F$.
\end{theorem}

\begin{proof}[Sketch of proof] As before, one implication is given by Theorem \ref{theorem: ST->GT}. Combining Theorem \ref{theorem: classification w=1}, Theorem \ref{theorem: truth ST axioms}, and Theorem \ref{theorem: truth MTAST}, we obtain a list 55 putative Sato-Tate groups. Making effective the proof of Theorem \ref{theorem: ST->GT}, from the 55 putative Sato-Tate groups, one obtains 55 different putative Galois endomorphism types (see \cite[Table 8]{FKRS12}). For 52 of these putative Galois endomorphism types one finds abelian surfaces realizing them (see \cite[Table 11]{FKRS12}; in fact, one can find genus 2 curves\footnote{Numerical tests of the Sato-Tate Conjecture have been performed for these 52 genus 2 curves. To visualize the matching between the numerical tests and the theoretical predictions see the animations at
\url{http://math.mit.edu/~drew}.} whose Jacobians realize the Galois endomorphism types). The remaining 3 Galois endomorphism types correspond to structures of endomorphism algebras of abelian varieties that were showed not to exist in the work of Shimura (see \cite[\S4.4]{FKRS12}). 
\end{proof}

\begin{remark}
Not all possibilities of Theorem~\ref{theorem: STgroupsg1} and Theorem~\ref{theorem: STgroupsg2} arise when we fix the ground field $F$. For example, only 2 Sato-Tate groups arise for elliptic curves defined over $\Q$, and the Sato-Tate group of an abelian surface defined over~$\Q$ is restricted to a list of~$34$ possibilities, all of which can occur for a particular choice of $A/\Q$ (see \cite[Thm. 4.3]{FKRS12}). 
\end{remark}

At the moment, we lack of a theorem of the style of the two previous ones for $g=3$. A preliminary computation has shown, however, that several hundreds of Sato-Tate groups and Galois endomorphism types arise in dimension 3. We finish by mentioning some results establishing the validity of the conjecture in some particular cases. 

\begin{remark}[g=2]
Among the 34 Galois endomorphism types that arise for $F=\Q$,~$18$ correspond to abelian surfaces $A/\Q$ such that $\End(A)$ strictly contains~$\Z$ (this happens if and only if either $A$ is isogenous over $\Q$ to the square of an elliptic curve or if $A$ is of $\GL_2-$type). Building on the description given by Ribet \cite{Rib92} of the Tate module of an abelian variety of $\GL_2-$type, J. Gonz\'alez \cite{Gon11} has proved equidistribution of the normalized local factors in~15 of these~18 cases.
Among the~34 Sato-Tate groups that arise over $\Q$, $18$ have connected component isomorphic to $\Unitary(1)$. Each of these groups can be achieved by considering the Jacobian of a twist of either $y^2=x^5-x$ or $y^2=x^6 +1$. For such Jacobians the Sato-Tate Conjecture is known to hold (see \cite{FS12}). We note that not all such Jacobians are of $\GL_2$-type, and therefore not all these cases are covered by \cite{Gon11}. Recently C. Johansson \cite{Joh13} has proved that the Sato-Tate Conjecture holds true for many \emph{nongeneric}\footnote{Recall that one says that an abelian surface $A$ is generic if $\End(A_{\Qbar})=\Z$, or if, equivalently, $\ST(A)=\USp(4)$.} cases of abelian surfaces.
\end{remark}

\begin{remark}[$g\geq 3$]
The Sato-Tate Conjecture has been proved to hold for Jacobians of twists of the Fermat and the Klein curves (see \cite{FLS12}). These are curves of genus~3. One obtains~48 different distributions when considering twists of the Fermat curve, and~22 when considering twists of the Klein curve.
In \cite{FGL13}, the Frobenius distributions and Sato-Tate groups of non-degenerate quotients of Fermat curves of prime exponent are computed (the Jacobians of such curves are abelian varieties with complex multiplication of arbitrarily large dimension). 
\end{remark}


\begin{thebibliography}{McK-Sta}

\bibitem[AS12]{AS12}
O. Ahmadi, I.E. Shparlinski, \emph{On the distribution of the number of points on algebraic curves in extensions of finite fields}, Math. Res. Lett. \textbf{17} (2012), no. 04, 689--699.

\bibitem[BK14]{BK11}
G. Banaszak and K.S. Kedlaya, \emph{An algebraic Sato-Tate group and Sato-Tate conjecture}, to appear in the Indiana University Mathematics Journal (2014).

\bibitem[BK12]{BK12} A.I. Bucur and K.S. Kedlaya \emph{The effective Sato-Tate Conjecture and applications}, available at arXiv:1301.0139 (2013).

\bibitem[BLGHT11]{BLGHT11} T. Barnet-Lamb, D. Geraghty, M. Harris, and R. Taylor,
\emph{A family of Calabi-Yau varieties and potential automorphy II},
Publ. RIMS \textbf{47} (2011), 29--98.

\bibitem[Del71]{Del71} P. Deligne, \emph{Formes modulaires et repr\'esentations $\ell$-adiques}, S\'eminaire Bourbaki no 355, in Lecture Notes in Mathematics, Springer, \textbf{179}, 1971, 139--172.

\bibitem[Del74]{Del74}
P. Deligne, \emph{La conjecture de Weil. I.}, Publications Math\'ematiques de l'IH\'ES \textbf{43} (1974), 273--307.

\bibitem[Del82]{Del82} P. Deligne, \emph{Hodge cycles on abelian varieties (notes by J.S. Milne)}, Lecture Notes in Mathematics \textbf{900} (1982), 9--100.

\bibitem[FGL14]{FGL13}
F. Fit\'e, J. Gonz\'alez, and J-C. Lario, \emph{Frobenius distribution for quotients of Fermat curves of prime degree}, available at arXiv:1403.0807 (2014).

\bibitem[FH91]{FH91} W. Fulton, J. Harris, \emph{Representation Theory. A First course}, Springer-Verlag New York Inc., 1991. 

\bibitem[FKRS12]{FKRS12} F.Fit\'e, K.S. Kedlaya, A.V. Sutherland, V. Rotger, \emph{Sato-Tate distributions and Galois endomorphism modules in genus 2}, Compositio Mathematica \textbf{148}, n. 5 (2012), 1390--1442. 

\bibitem[FKS12]{FKS12} F.Fit\'e, K.S. Kedlaya, and A.V. Sutherland, \emph{Sato-Tate groups of some weight 3 motives}, available at arXiv:1212.0256 (2012).

\bibitem[FLS14]{FLS12} F. Fit\'e, E. Lorenzo, and A.V. Sutherland, \emph{Sato-Tate distributions of twists of the Fermat and Klein curves}, Preprint (2014).

\bibitem[FS14]{FS12} 
F. Fit\'e and A.V. Sutherland, \emph{Sato-Tate distributions of twists of $y^2=x^5-x$ and $y^2=x^6+1$}, to appear in Algebra \& Number Theory.

\bibitem[Gon14]{Gon11} J. Gonz\'alez, \emph{The Frobenius traces distribution for modular abelian surfaces}, The Ramanujan Journal \textbf{33}, No. 2 (2014), 247--261.

\bibitem[Hec20]{Hec20} E. Hecke, \emph{Eine neue Art von Zetafunktionen und ihre Beziehungen zur Verteilung der Primzahlen. Zweite Mitteilung}, Math. Zeit. \textbf{6} (1920), 11--51.

\bibitem[Joh13]{Joh13} C. Johansson, \emph{On the Sato-Tate conjecture for non-generic abelian surfaces}, available at arXiv:1307.6478 (2013).

\bibitem[Kat13]{Ka13} N. Katz, \emph{Sato-Tate in the higher dimensional case: elaboration of 9.5.4 in Serre's $N_X(p)$ book}, L'enseignement math\'ematique \textbf{59}, Issue 3/4 (2013), 359--377. 

\bibitem[KW09]{KW09} C. Khare, J.-P. Wintenberger \emph{Serre's modularity conjecture (I)}, Invent. math. \textbf{178} (2009), 485--504.

\bibitem[KS09] {KS09} K.S. Kedlaya, A.V. Sutherland, \emph{Hyperelliptic curves, $L$-polynomials, and random matrices}, Arithmetic, Geometry, Cryptography, and Coding Theory (AGCT 2007), Contemporary Math. {\bf 487}, Amer. Math. Soc., 2009, 119--162.

\bibitem[Kub65]{Kub65} T. Kubota, \emph{On the field extension by complex multiplication}, Transactions of the American Mathematical Society \textbf{118} (1965), 113--122.

\bibitem[Lan94]{La94} S. Lang, \emph{Algebraic Number theory}, Springer, 1994.

\bibitem[MM]{MM} K. Murty, R. Murty, \emph{The Sato-Tate conjecture and generalizations}, Current trends in Science, Platinum Jubilee Special, Indian Academy of Sciences, 2009.

\bibitem[MY11]{MT11} R. Masri, T. Yang \emph{Nonvanishing of Hecke $L$-functions for CM fields and ranks of abelian varieties}, Geometric and Functional Analysis \textbf{21} (2011), 648--679. 

\bibitem[Mum69]{Mum69} D. Mumford, {\em A note on Shimura's paper ``Discontinuous subgroups and abelian varieties''}, Math. Ann. {\bf 181} (1969), 345--351.

\bibitem[Neu92]{Neu92} J. Neukirch, \emph{Algebraische Zahlentheorie}, Springer-Verlag Berlin-Heidelberg-New York, 1992.

\bibitem[Rib77]{Rib77} K. Ribet, \emph{Galois representations attached to eigenforms with Nebentypus}, in: J.-P. Serre, D. B. Zagier (eds.), \emph{Modular Functions of one Variable V} (Bonn 1976), Lect. Notes in Math. \textbf{601}, Springer (1977), 17--52.

\bibitem[Rib80]{Rib80} K.A. Ribet, \emph{Division fields of abelian varieties with complex multiplication}, Soci\'et\'e Math\'ematique de France, 2e s\'erie, M\'emoire No. 2, 1980, p. 75--94.

\bibitem[Rib83]{Rib83} K.A. Ribet, \emph{Hodge Classes on Certain Types of Abelian Varieties}, American Journal of Mathematics \textbf{105}, No. 2 (1983), 523--538.

\bibitem[Rib92]{Rib92} K.A. Ribet, \emph{Abelian varieties over $\Q$ and modular forms},
Algebra and topology 1992 ({T}aej\u on), Korea Adv. Inst.
Sci. Tech., {T}aej\u on, 1992, 53--79.

\bibitem[Ser66]{Ser66} J.-P. Serre, \emph{Lettre \`a Armand Borel du 18/5/1966}.

\bibitem[Ser68]{Ser68} J.-P. Serre, {\em Abelian $\ell$-adic Representations and Elliptic Curves}, W.A. Benjamin Inc., 1968.

\bibitem[Ser72]{Ser72} J.-P. Serre, \emph{Propri\'et\'es galoisiennes des points d'ordre fini des courbes elliptiques}, Inventiones math. \textbf{15}, 259--331 (1972).

\bibitem[Ser77]{Ser77} J.-P. Serre, \emph{Representations $\ell$-adiques}, in S. Iyanaga (ed.), Algebraic Number Theory, Japan Society for the Promotion of Science, Kyoto University Press, 1977, 177--193.

\bibitem[Ser95]{Ser95} J.-P. Serre, \emph{Propri\'et\'es conjecturales des groupes de Galois motiviques et des repr\'esentations
l-adiques}, Motives (Seattle, WA, 1991), Proceedings of Symposia in Pure Math. \textbf{55},
Amer. Math. Soc., 1994, 377--400.

\bibitem[Ser12]{Ser12}
J.-P. Serre, \textit{Lectures on $N_X(p)$}, A.K. Peters, 2012.

\bibitem[Shi98]{Shi98} G. Shimura, \emph{Abelian varieties with complex multiplication and modular forms}, Princeton Univ. Press, 1998.

\bibitem[Sil94]{Sil94} J.H. Silverman, \emph{Advanced Topics in the Arithmetic of Elliptic Curves},
Graduate Texts in Math. \textbf{151}, Springer, 1994.

\bibitem[Yu13]{Yu13} C.-F. Yu, \emph{Mumford-Tate Conjecture for CM abelian varieties}, preprint, 2013.

\bibitem[Zar14]{Zar14} Y.G. Zarhin, \emph{Eigenvalues of Frobenius endomorphisms of abelian varieties}, available at arXiv:1312.0377 (2014).

\end{thebibliography}
\end{document}